\documentclass[a4paper, 11pt]{article}

\usepackage[margin=1in]{geometry}
\usepackage[T1]{fontenc}
\usepackage{lmodern}
\usepackage{microtype}

\usepackage{graphicx}
\usepackage{appendix}
\usepackage[round]{natbib}
\usepackage[textsize=footnotesize]{todonotes}

\usepackage{algorithm}
\usepackage{algorithmic}

\usepackage{amsmath}
\usepackage{amssymb}
\usepackage{mathtools}
\usepackage{amsthm}
\usepackage{bm}

\DeclareMathOperator*{\argmin}{argmin}
\newcommand{\ttimes}[2]{\times_{#1}^{#2}}

\theoremstyle{plain}
\newtheorem{theorem}{Theorem}[section]
\newtheorem{proposition}[theorem]{Proposition}

\theoremstyle{definition}

\theoremstyle{remark}

\usepackage{booktabs}
\usepackage{siunitx}
\usepackage{multirow}
\usepackage{makecell}
\usepackage[font=small, labelfont=bf]{caption}

\usepackage{subcaption}
\usepackage{enumitem}
\usepackage{enumitem}
\setlist[itemize]{itemsep=3pt, topsep=5pt, parsep=0pt, partopsep=0pt, leftmargin=*}

\usepackage{tikz}
\usetikzlibrary{positioning}

\tikzset{
    tensor/.style={circle, fill=black, inner sep=3.5pt},
    bond/.style={draw, line width=1.5pt},
    diagonal_tensor_r/.style={
        circle,
        draw,
        line width=1.5pt,
        inner sep=3.5pt,
        path picture={
            \draw[bond] (path picture bounding box.south west) -- (path picture bounding box.north east);
        }
    },
    diagonal_tensor_l/.style={
        circle, 
        draw, 
        line width=1.5pt, 
        inner sep=3.5pt,
        path picture={
            \draw[bond] (path picture bounding box.south east) -- (path picture bounding box.north west);
        }
    },
    orthogonal_tensor_r/.style={
        diagonal_tensor_r,
        path picture={
            \fill[black] 
                (path picture bounding box.south west) -- 
                (path picture bounding box.north west) -- 
                (path picture bounding box.north east) -- 
                cycle;
            \draw[bond] (path picture bounding box.south west) -- (path picture bounding box.north east);
        }
    },
    orthogonal_tensor_l/.style={
        diagonal_tensor_l,
        path picture={
            \fill[black] 
                (path picture bounding box.south east) -- 
                (path picture bounding box.north east) -- 
                (path picture bounding box.north west) -- 
                cycle;
            \draw[bond] (path picture bounding box.south east) -- (path picture bounding box.north west);
        }
    }
}

\usepackage[
    colorlinks=true,
    linkcolor=blue,
    citecolor=blue,
    urlcolor=blue,
]{hyperref}

\title{Online TT-ALS for Streaming Tensor Decomposition with Incremental Orthogonalization}
\author{
    Hiroki Takeda\thanks{Department of Pure and Applied Mathematics, The University of Osaka, Japan} \and
    Yuto Miyatake\thanks{D3 Center, The University of Osaka, Japan} \and
    Daisuke Furihata\footnotemark[2]
}
\date{}

\begin{document}
\maketitle

\begin{abstract}
Tensor Train (TT) decomposition is a powerful technique for analyzing high-dimensional data.
Existing algorithms for computing TT decompositions can be categorized into two main types: conventional batch-based approaches and recursive online methods.
In the context of streaming data, batch methods typically achieve higher reconstruction accuracy but often suffer from memory exhaustion, while online methods provide greater computational efficiency.
In this work, we introduce Online TT-ALS (Alternating Least Squares), an algorithm that sequentially enforces orthogonality constraints.
This approach allows for efficient and exact updates of the core tensor while maintaining high reconstruction accuracy.
Theoretically, we prove that enforcing these orthogonal gauge constraints guarantees monotonic decrease of the local objective function and temporal smoothness.
Computationally, our deterministic single-sweep update reduces the rank dependence from quadratic to linear, achieving an overall complexity of $\mathcal{O}(I^{n-1} r)$.
Experimental results demonstrate that the proposed method outperforms existing online techniques not only in terms of mathematical approximation accuracy but also in human perception-based video quality metrics.
Furthermore, compared to recent deep learning-based paradigms, our algebraic approach achieves speedups of several orders of magnitude.
Consequently, our method exhibits high computational efficiency and is suitable for low-latency real-time processing applications.
\end{abstract}

\section{Introduction}
In recent years, the analysis of high-dimensional data has become increasingly critical across a wide range of fields.
High-dimensional data are generated in numerous applications, including image and video processing, recommendation systems, and sensor networks, creating a demand for efficient methods to extract useful information from such data.
Tensor decomposition has emerged as a powerful tool for decomposing high-dimensional data into low-dimensional latent structures and is widely used for dimensionality reduction and feature extraction \citep{Tensor-Decomposition}.
In particular, Tensor Train (TT) decomposition \citep{Tensor-Train-Decomposition} represents a high-order tensor as a chain product of third-order tensors.
A key advantage of this format is that the number of parameters grows linearly with the tensor order, thereby mitigating the curse of dimensionality and facilitating the handling of large-scale tensors.
Owing to this efficient representation capability, the TT format has been successfully applied to practical large-scale problems, such as compressing the massive parameters of deep neural networks \citep{Novikov-NeurIPS2015}.
This demonstrates that the TT format possesses high expressive power, enabling the capture of essential information from high-dimensional data within a highly compact structure.

In many practical scenarios, data are generated continuously as a stream, rendering batch processing inefficient or infeasible.
In the context of matrix analysis, incremental low-rank approximation techniques, such as Incremental PCA \citep{IVT-Ross}, have been well established for tasks such as robust visual tracking.
Similarly, for higher-order tensors, online algorithms for CP and Tucker decompositions have been extensively studied to handle time-evolving data \citep{DTA-Sun, OnlineCP-Mardani}.

Following this trend, online TT decomposition algorithms using recursive update schemes have also been developed to reduce the high computational costs of batch processing \citep{Survey-Streaming-Tensor-Decomposition}.
While these methods enhance real-time processing efficiency,
they often sacrifice reconstruction accuracy compared to batch methods by not enforcing structural constraints such as orthogonality.
This 
can lead to ill-conditioned updates, affecting numerical stability and reconstruction fidelity.
For data exhibiting temporal dependencies, such as video streams, failing to adequately capture correlation information can 
result in blurred footage or inaccurate motion representation.

To overcome these challenges, we propose a new online TT decomposition algorithm for streaming data, which we refer to as Online TT-ALS (Alternating Least Squares).
In a streaming environment, our method imposes orthogonality constraints on each core tensor to streamline matrix computations and updates the cores based on the exact solution of the local optimization subproblems.

The main contributions of this paper are summarized as follows:
\begin{itemize}
    \item \textbf{Theoretical Guarantees:} We establish the theoretical foundation for our online TT decomposition, proving that enforcing orthogonal gauge constraints guarantees monotonic decrease of the local objective function and temporal smoothness.
    \item \textbf{Linear Computational Complexity:} We demonstrate that our deterministic single-sweep update achieves a computational complexity of $\mathcal{O}(I^{n-1} r)$.
    By scaling linearly rather than quadratically with respect to the TT-rank, our method reduces computational overhead compared to existing first-order approximations.
    \item \textbf{Empirical Scalability and Low Latency:} Through extensive experiments, we validate that the proposed method scales to high-order streaming tensors where batch methods fail due to memory exhaustion.
    Furthermore, compared to deep learning-based continuous functional approximations, our algebraic method achieves speedups of several orders of magnitude, providing a highly practical solution for low-latency real-time processing.
\end{itemize}

\subsection{Related Work}
Algorithms for TT decomposition can be broadly categorized into two classes: batch methods and online methods.

The most fundamental batch algorithm is TT-SVD \citep{Tensor-Train-Decomposition}.
This method constructs the TT cores deterministically by performing a sequence of singular value decompositions (SVD) on the unfolding matrices of the input tensor.
While TT-SVD guarantees a quasi-optimal approximation with a theoretical error bound, it requires the entire tensor to be stored in memory and involves computationally intensive full SVD operations.
Consequently, it is often impractical for large-scale high-dimensional data.

To improve computational efficiency for large tensors, iterative batch methods have been developed.
TT-ALS \citep{TT-ALS} is a representative method in this category that computes the TT decomposition based on an Alternating Linear Scheme, which extends the concept of Alternating Least Squares.
TT-ALS efficiently determines the TT decomposition by optimizing each core tensor sequentially while fixing the others.
When updating a specific core, computational efficiency and numerical stability are achieved by imposing left-orthogonality constraints on the left cores and right-orthogonality constraints on the right cores.
Specifically, the algorithm performs a sweeping process, updating each core sequentially from left to right and then from right to left.
This bidirectional sweeping is repeated until a predefined convergence criterion is met.
\citet{TT-ALS} theoretically proved that the cost function decreases monotonically during this iterative process.

The effectiveness of such batch TT decomposition methods has been demonstrated in various computer vision tasks.
For instance, \citet{TT-Video-Completion} employed optimization methods based on the TT format for image and video completion problems, showing that the TT representation can effectively capture global correlations and spatiotemporal features of high-dimensional visual data.
However, this iterative update scheme necessitates storing the entire tensor and making multiple passes over the data.
Consequently, when applied to streaming data or large-scale high-dimensional tensors, the memory requirements become prohibitive, often leading to out-of-memory (OOM) failures and making real-time processing impossible.

To circumvent the limitations of batch processing, online TT decomposition algorithms using recursive update schemes have been proposed.
TT-FOA \citep{TT-FOA} is an online TT decomposition algorithm designed for streaming data.
It employs a recursive scheme to minimize a weighted least squares error, incorporating a forgetting factor to account for historical information.
While TT-FOA successfully reduces the memory footprint compared to batch methods, its computational complexity scales quadratically with respect to the TT-rank (i.e., $\mathcal{O}(I^{n-1} r^2)$).
Furthermore, it typically applies block coordinate descent to update the cores without explicitly maintaining the orthogonal structure of the TT format.
As a result, it suffers from numerical instability due to ill-conditioned Gram matrices and inherently requires a warm-up period for the approximation error to converge.
In applications like video tracking, this lack of structural constraints yields lower reconstruction accuracy than batch methods, resulting in artifacts such as blurring.

Beyond traditional algebraic TT decompositions, deep learning-based continuous functional approximations, such as OFTD \citep{OFTD}, have recently gained traction for streaming tensor representation.
These methods leverage neural networks to model complex spatial-temporal correlations, achieving high performance particularly in tensor completion tasks with missing data.
However, they operate under a fundamentally different computational paradigm.
Updating the network weights for each newly arriving frame requires highly iterative forward and backward passes.
The computational overhead of this backpropagation process renders these methods impractical for dense streaming applications that require immediate state updates.

Compared to these existing approaches, our proposed Online TT-ALS method performs matrix computations efficiently by sequentially enforcing orthogonality constraints and updates the cores using exact solutions to the local optimization problems.

The remainder of this paper is organized as follows.
Section \ref{sec:preliminaries} introduces the fundamental concepts, notation, and necessary background on Tensor Train decomposition and tensor orthogonality.
Section \ref{sec:proposed_method} details our proposed Online TT-ALS algorithm, including the problem formulation and the exact core update process via sequential orthogonalization.
Section \ref{sec:theoretical_and_computational_analysis} establishes the theoretical foundation of our approach, providing proofs for monotonic decrease of the local objective function and an analysis of the linear computational complexity.
Section \ref{sec:experiments} presents comprehensive experimental results, validating the scalability, long-term tracking stability, and low-latency performance of our method against existing baselines.
Finally, Section \ref{sec:conclusion} concludes the paper and outlines potential directions for future research.

\section{Preliminaries}
\label{sec:preliminaries}
\subsection{Notation}
Scalars are denoted by lowercase letters (e.g., $a, b, c$), vectors by bold lowercase letters (e.g., $\bm{a}, \bm{b}, \bm{c}$), matrices by bold uppercase letters (e.g., $\mathbf{A}, \mathbf{B}, \mathbf{C}$), and tensors by calligraphic letters (e.g., $\mathcal{A}, \mathcal{B}, \mathcal{C}$).
The Frobenius norm and the Euclidean $\ell_2$-norm are denoted by $\|\cdot\|_\mathrm{F}$ and $\|\cdot\|_2$, respectively.

For a tensor $\mathcal{X} \in \mathbb{R}^{I_1 \times I_2 \times \cdots \times I_n}$, mode permutation is defined as an operation that rearranges the order of the tensor modes.
For instance, applying a mode permutation that swaps mode 1 and mode 3 to $\mathcal{X} \in \mathbb{R}^{I_1 \times I_2 \times I_3}$ yields $\mathcal{X}_{(3,2,1)} \in \mathbb{R}^{I_3 \times I_2 \times I_1}$.
Specifically, the elements satisfy the following relationship: $\mathcal{X}_{(3,2,1)}(i_3, i_2, i_1) = \mathcal{X}(i_1, i_2, i_3)$.

A tensor can be 
represented in matrix form via matricization (also known as unfolding).
For example, combining the first two modes of $\mathcal{X} \in \mathbb{R}^{I_1 \times I_2 \times I_3}$ into the row indices yields a matrix $\mathbf{X} \in \mathbb{R}^{(I_1 I_2) \times I_3}$, where
\begin{equation*}
\mathbf{X}(j, k) = \mathcal{X}(i_1, i_2, i_3), \quad j = i_1 + (i_2 - 1)I_1, \quad k = i_3,
\end{equation*}
for $1 \leq j \leq I_1 I_2$ and $1 \leq k \leq I_3$.
The original tensor can be recovered by decomposing the row index $j$, an operation referred to as tensorization (also known as folding).

Given two tensors $\mathcal{A} \in \mathbb{R}^{I_1 \times I_2 \times \cdots \times I_n}$ and $\mathcal{B} \in \mathbb{R}^{J_1 \times J_2 \times \cdots \times J_m}$ with $I_n = J_1$, the mode product (or contraction) $\mathcal{C} = \mathcal{A} \ttimes{n}{1} \mathcal{B} \in \mathbb{R}^{I_1 \times \cdots \times I_{n-1} \times J_2 \times \cdots \times J_m}$ is defined by
\begin{equation*}
\begin{split}
&\mathcal{C}(i_1, \ldots, i_{n-1}, j_2, \ldots, j_m) \\
&\quad = \sum_{s=1}^{I_n} \mathcal{A}(i_1, \ldots, i_{n-1}, s) \mathcal{B}(s, j_2, \ldots, j_m).
\end{split}
\end{equation*}
Mode products for other modes can be defined similarly.
When the modes involved in the product are clear from context, the notation may be simplified to $\mathcal{C} = \mathcal{A} \times \mathcal{B}$.

Furthermore, for tensors $\mathcal{A} \in \mathbb{R}^{I_1 \times I_2 \times \cdots \times I_n}$ and $\mathcal{B} \in \mathbb{R}^{I_1 \times I_2 \times \cdots \times I_{n-1} \times 1}$, the concatenation $\mathcal{C} = \mathcal{A} \boxplus \mathcal{B}$ is defined as
\begin{equation*}
\begin{split}
&\mathcal{C}(i_1, i_2, \ldots, i_{n-1}, i_n) \\
&\quad = \begin{cases}
\mathcal{A}(i_1, i_2, \ldots, i_{n-1}, i_n) & (1 \leq i_n \leq I_n), \\
\mathcal{B}(i_1, i_2, \ldots, i_{n-1}, 1) & (i_n = I_n + 1).
\end{cases}
\end{split}
\end{equation*}
This operation increases the size of the $n$-th mode, resulting in $\mathcal{C} \in \mathbb{R}^{I_1 \times I_2 \times \cdots \times I_{n-1} \times (I_n + 1)}$.

\subsection{Tensor Train Decomposition}
Let $\mathcal{X} \in \mathbb{R}^{I_1 \times I_2 \times \cdots \times I_n}$ be an $n$-th order tensor.
The TT decomposition expresses $\mathcal{X}$ as a chain product of third-order tensors: 
\begin{equation*}
\mathcal{X} = \mathcal{G}_1 \ttimes{2}{1} \mathcal{G}_2 \ttimes{3}{1} \cdots \ttimes{n}{1} \mathcal{G}_n,
\end{equation*}
where $\mathcal{G}_k \in \mathbb{R}^{r_{k-1} \times I_k \times r_k}$ is the $k$-th TT core with the usual assumption that $r_0 = r_n = 1$.
The tuple $(r_0, r_1, \ldots, r_n)$, or simply $(r_1, \ldots, r_{n-1})$, is referred to as the TT-rank.
In this paper, we employ tensor diagrams to visualize tensor operations.
For example, Figure~\ref{fig:tensor-diagram-ttd} illustrates the TT decomposition.
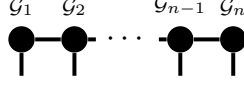
\begin{figure}[t]
\centering
\begin{tikzpicture}[node distance=0.7cm, on grid]
\node[tensor, label={above: \scriptsize{$\mathcal{G}_1$}}] (G1) {};
\node[tensor, right=of G1, label={above: \scriptsize{$\mathcal{G}_2$}}] (G2) {};
\node[right=0.7cm of G2] (dots) {\large $\cdots$};
\node[tensor, right=0.7cm of dots, label={above: \scriptsize{$\mathcal{G}_{n-1}$}}] (Gn-1) {};
\node[tensor, right=of Gn-1, label={above: \scriptsize{$\mathcal{G}_n$}}] (Gn) {};
\draw[bond] (G1) -- (G2) -- (dots) -- (Gn-1) -- (Gn);
\foreach \n in {G1, G2, Gn-1, Gn} {
\draw[bond] (\n.south) -- +(0,-0.3);
}
\end{tikzpicture}
\caption{Tensor diagram of TT decomposition.}
\label{fig:tensor-diagram-ttd}
\end{figure}
In this graphical notation, nodes represent tensors, while edges represent tensor modes.
A connected edge between two nodes indicates a tensor contraction, or a mode product, along with the corresponding indices.

\subsection{Streaming Tensors}
A streaming tensor is defined as a tensor to which data is sequentially appended over time.
Specifically, let us consider a tensor $\mathcal{X}[t] \in \mathbb{R}^{I_1 \times I_2 \times \cdots \times I_{n-1} \times I_n[t]}$.
In this model, the data grows along the $n$-th mode over time, where $I_n[t]$ denotes the size of the $n$-th mode at time step $t$.
Let $\overline{\mathcal{X}}_t \in \mathbb{R}^{I_1 \times I_2 \times \cdots \times I_{n-1}}$ denote the new data slice arriving at time $t$.
The tensor $\mathcal{X}[t]$ can then be expressed as:
\begin{equation*}
\mathcal{X}[t] = \mathcal{X}[t-1] \boxplus \overline{\mathcal{X}}_t.
\end{equation*}
In this paper, we assume that $I_n[t] = I_n[t-1] + 1$, implying that a single slice of data arrives at each time step.

\subsection{Tensor Orthogonality}
This section introduces the concept of orthogonality for third-order tensors, which is essential for the subsequent discussions.
Let $\mathcal{X} \in \mathbb{R}^{P \times Q \times R}$ be a third-order tensor.
The tensor $\mathcal{X}$ is said to be left- or right-orthogonal if its matricization $\mathbf{X} \in \mathbb{R}^{(P Q) \times R}$ or $\mathbf{X} \in \mathbb{R}^{P \times (Q R)}$, respectively, satisfies
\begin{equation*}
\mathbf{X}^\top \mathbf{X} = I \in \mathbb{R}^{R \times R}
\quad \text{or, resp.,} \quad
\mathbf{X} \mathbf{X}^\top = I \in \mathbb{R}^{P \times P}.
\end{equation*}
Figure~\ref{fig:tensor-diagram-orthogonality} illustrates the left- and right-orthogonality using tensor diagrams.
\begin{figure}[t]
  \centering
  \begin{tikzpicture}
    \node[orthogonal_tensor_l, label={above:Left-orth.}] (X) at (0,0) {};
    \draw[bond] (X.west) -- (-0.5,0);
    \draw[bond] (X.east) -- (0.5,0);
    \draw[bond] (X.south) -- (0,-0.5);
    \node[orthogonal_tensor_r, label={above:Right-orth.}] (Y) at (2,0) {};
    \draw[bond] (Y.west) -- (1.5,0);
    \draw[bond] (Y.east) -- (2.5,0);
    \draw[bond] (Y.south) -- (2,-0.5);
  \end{tikzpicture}
\caption{Tensor diagrams for left- and right-orthogonality.}
\label{fig:tensor-diagram-orthogonality}
\end{figure}
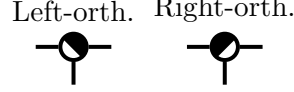

\section{Proposed Method}
\label{sec:proposed_method}
In this paper, we propose Online TT-ALS (Alternating Least Squares), a Tensor Train decomposition algorithm for streaming data.
Our method imposes orthogonality constraints on the TT cores to avoid the computational bottlenecks associated with matrix operations in existing online methods.
As a result, this approach allows for more efficient updates of the cores.
The proposed method can be viewed as an online extension of the batch TT-ALS algorithm.

\subsection{Problem Formulation}
Let $\mathcal{X}[t] \in \mathbb{R}^{I_1 \times I_2 \times \cdots \times I_{n-1} \times I_n[t]}$ denote the streaming tensor at time $t$.
We express the TT decomposition of $\mathcal{X}[t]$ as follows:
\begin{equation*}
  \mathcal{X}[t] = \mathcal{G}_1[t] \ttimes{2}{1} \mathcal{G}_2[t] \ttimes{3}{1} \cdots \ttimes{n}{1} \mathcal{G}_n[t],
\end{equation*}
where $\mathcal{G}_k[t] \in \mathbb{R}^{r_{k-1} \times I_k \times r_k}$ is the $k$-th TT core at time $t$.
We assume the boundary conditions $r_0 = r_n = 1$.
Here, for $1 \leq k \leq n-1$, the dimensions of the TT cores $\mathcal{G}_k[t] \in \mathbb{R}^{r_{k-1} \times I_k \times r_k}$ remain fixed over time, whereas the size of the last core $\mathcal{G}_n[t] \in \mathbb{R}^{r_{n-1} \times I_n[t] \times 1}$ grows with time $t$.
These TT cores are estimated by solving the following optimization problem:
\begin{equation}
    \begin{split}
        & \{\mathcal{G}_1[t], \mathcal{G}_2[t], \ldots, \mathcal{G}_n[t]\} \\
        & \quad = \argmin_{\{\mathcal{G}_k\}_{k=1}^{n}} \left\| \mathcal{X}[t] - \mathcal{G}_1 \ttimes{2}{1} \mathcal{G}_2 \ttimes{3}{1} \cdots \ttimes{n}{1} \mathcal{G}_n \right\|_\mathrm{F}^2. 
    \end{split}\label{eq:proposed-objective-full}
\end{equation}
By considering the constituent data slices,
the minimization problem~\eqref{eq:proposed-objective-full} can be rewritten as:
\begin{equation}
    \begin{split}
        &\!\!\!\!\{\mathcal{G}_1[t], \mathcal{G}_2[t], \ldots, \mathcal{G}_n[t]\} \\
        &= \argmin_{\{\mathcal{G}_k\}_{k=1}^{n}} \sum_{i=1}^{I_n[t]} \left\| \overline{\mathcal{X}}_i - \mathcal{G}_1 \ttimes{2}{1} \mathcal{G}_2 \ttimes{3}{1} \cdots \ttimes{n}{1} \bm{g}_n[i] \right\|_\mathrm{F}^2, 
    \end{split}\label{eq:proposed-objective-slice}
\end{equation}
where $\bm{g}_n[i] \in \mathbb{R}^{r_{n-1}}$ denotes the $i$-th column of $\mathcal{G}_n \in \mathbb{R}^{r_{n-1} \times I_n[t]}$. 
Note that in the minimization problem~\eqref{eq:proposed-objective-slice}, $\overline{\mathcal{X}}_i \in \mathbb{R}^{I_1 \times I_2 \times \cdots \times I_{n-1}}$ denotes the $i$-th data slice.

However, in a streaming environment, retaining data slices from all past time steps is often infeasible due to memory constraints.
Consequently, minimizing the sum over all terms in \eqref{eq:proposed-objective-slice} is computationally impractical.

\subsection{Key Idea of Online TT-ALS}
Let us describe the key idea underlying the proposed online TT-ALS method.
We focus on the term corresponding to the current time $t$ in the minimization problem~\eqref{eq:proposed-objective-slice}.
Specifically, given the data slice $\overline{\mathcal{X}}_t$ arriving at time $t$, we aim to sequentially minimize the following objective function for $k=1, 2, \ldots, n$, while keeping all other TT cores fixed:
\begin{equation}
    f(\mathcal{G}_k) = \frac{1}{2}\left\| \overline{\mathcal{X}}_t - \mathcal{A}_k[t] \ttimes{k}{1} \mathcal{G}_k \ttimes{k+1}{1} \mathcal{B}_k[t-1] \right\|_\mathrm{F}^2. \label{eq:proposed-objective}
\end{equation}
Here, the terms $\mathcal{A}_k[t] \in \mathbb{R}^{I_1 \times I_2 \times \cdots \times I_{k-1} \times r_{k-1}}$ and $\mathcal{B}_k[t-1] \in \mathbb{R}^{r_k \times I_{k+1} \times I_{k+2} \times \cdots \times I_n[t-1]}$ are defined as follows:
\begin{align*}
    \mathcal{A}_k[t] &= \mathcal{G}_1[t] \ttimes{2}{1} \mathcal{G}_2[t] \ttimes{3}{1} \cdots \ttimes{k-1}{1} \mathcal{G}_{k-1}[t], \\
    \begin{split}
        \mathcal{B}_k[t-1] &= \mathcal{G}_{k+1}[t-1] \ttimes{3}{1} \mathcal{G}_{k+2}[t-1] \ttimes{4}{1} \cdots \\
        &\phantom{=}\quad \ttimes{n-k}{1} \mathcal{G}_{n-1}[t-1] \ttimes{n-k+1}{1} \bm{g}_n[t-1].
    \end{split}
\end{align*}
The tensors $\mathcal{A}_k[t]$ and $\mathcal{B}_k[t-1]$ represent the contractions of the TT cores preceding and succeeding the $k$-th core, respectively.
We define the boundary terms $\mathcal{A}_1[t]$ and $\mathcal{B}_n[t-1]$ as scalars equal to 1.
Since our method updates the cores sequentially from $k=1$ to $n$, $\mathcal{A}_k[t]$ is constructed using the updated cores at time $t$, whereas $\mathcal{B}_k[t-1]$ is constructed using the cores from the previous time step $t-1$.

By incorporating the core group $\mathcal{B}_k[t-1]$ from time $t-1$ into the objective function, we can efficiently update each core while retaining past information, without the need to explicitly process all past data slices as required in the minimization problem~\eqref{eq:proposed-objective-slice}.
In the proposed approach, we update each core $\{\mathcal{G}_k[t]\}_{k=1}^{n-1}$ and $\bm{g}_n[t]$ by minimizing this objective function $f(\mathcal{G}_k)$ sequentially for $k = 1, 2, \ldots, n$.

\subsection{Orthogonality Constraints}
Let $\widehat{\mathcal{X}}_t \in \mathbb{R}^{(I_1 I_2 \cdots I_{k-1}) \times I_k \times (I_{k+1} I_{k+2} \cdots I_n[t-1])}$ be the reshaped third-order tensor of the data slice $\overline{\mathcal{X}}_t$.
Let $\mathbf{A}_k[t] \in \mathbb{R}^{(I_1 I_2 \cdots I_{k-1}) \times r_{k-1}}$ and $\mathbf{B}_k[t-1] \in \mathbb{R}^{r_k \times (I_{k+1} I_{k+2} \cdots I_n[t-1])}$ be the matricizations of $\mathcal{A}_k[t]$ and $\mathcal{B}_k[t-1]$, respectively.
The reshaped tensor $\widehat{\mathcal{X}}_t$ depends on $k$, but we treat it as an operator acting on matrices.
As a result, there is no need to allocate memory to store this tensor explicitly during computation; the original tensor can be used directly whenever its elements are accessed.
For this reason, we omit the index $k$ to avoid redundant notation.
Then, the objective function \eqref{eq:proposed-objective} can be rewritten as
\begin{equation}
    f(\mathcal{G}_k) = \frac{1}{2} \left\| \widehat{\mathcal{X}}_t - \mathbf{A}_k[t] \ttimes{k}{1} \mathcal{G}_k \ttimes{k+1}{1} \mathbf{B}_k[t-1] \right\|_\mathrm{F}^2. \label{eq:proposed-objective-matrix}
\end{equation}
Minimizing \eqref{eq:proposed-objective-matrix} with respect to $\mathcal{G}_k$ leads to a least-squares problem, which requires computing the inverses of the Gram matrices $\mathbf{A}_k[t]^\top \mathbf{A}_k[t]$ and $\mathbf{B}_k[t-1] \mathbf{B}_k[t-1]^\top$.
However, computing the inverses of these dense matrices at every step of the streaming process is not only computationally expensive but also potentially prone to numerical instability.
Regarding numerical instability, we observe that the condition numbers of these matrices can reach the order of $10^4$ to $10^5$ in the subsequent experiments, and may be even higher in other scenarios.
At such scales, computations in single precision, and especially in half precision, may be significantly affected by numerical errors.

To circumvent this issue, the existing batch method, TT-ALS \citep{TT-ALS}, imposes orthogonality constraints on the TT cores.
Inspired by this approach, our method introduces constraints to maintain the orthogonality of each core in the streaming setting.
This avoids the matrix inversion bottleneck and enables efficient core updates.
Specifically, during the update of the $k$-th core, we impose left-orthogonality constraints on the preceding cores $\{\mathcal{G}_i[t]\}_{i=1}^{k-1}$ and right-orthogonality constraints on the succeeding cores $\{\mathcal{G}_i[t-1]\}_{i=k+1}^{n-1}$ and $\bm{g}_n[t-1]$.
Under these orthogonality constraints, the following properties hold:
\begin{equation*}
    \mathbf{A}_k[t]^\top \mathbf{A}_k[t] = \mathbf{I}_{r_{k-1}}, \quad \mathbf{B}_k[t-1] \mathbf{B}_k[t-1]^\top = \mathbf{I}_{r_k}.
\end{equation*}

Under these constraints, the gradient of \eqref{eq:proposed-objective-matrix} with respect to $\mathcal{G}_k$ is found to be
\begin{align*}
  \frac{\partial f}{\partial \mathcal{G}_k} 
  &= \begin{multlined}[t]
        -\mathbf{A}_k[t]^\top \ttimes{2}{1} \widehat{\mathcal{X}}_t \ttimes{3}{1} \mathbf{B}_k[t-1]^\top \\
        + \mathbf{A}_k[t]^\top \mathbf{A}_k[t] \ttimes{2}{1} \mathcal{G}_k \ttimes{3}{1} \mathbf{B}_k[t-1] \mathbf{B}_k[t-1]^\top 
      \end{multlined} \\
  &= -\mathbf{A}_k[t]^\top \ttimes{2}{1} \widehat{\mathcal{X}}_t \ttimes{3}{1} \mathbf{B}_k[t-1]^\top + \mathcal{G}_k.
\end{align*}
Therefore, the optimal $\mathcal{G}_k$ that minimizes $f(\mathcal{G}_k)$ is given by
\begin{equation}
  \mathcal{G}_k = \mathbf{A}_k[t]^\top \ttimes{2}{1} \widehat{\mathcal{X}}_t \ttimes{3}{1} \mathbf{B}_k[t-1]^\top. \label{eq:proposed-core-update}
\end{equation}
The algorithm used to preserve these orthogonality constraints throughout the update process is described in the next subsection.

\subsection{Algorithm Description}
The proposed method is summarized in Algorithm~\ref{alg:online-tt-als}.
First, we initialize the TT cores $\{\mathcal{G}_k[0]\}_{k=1}^{n-1}$ and $\bm{g}_n[0]$ with random values.
Below, we focus on the update procedure at time $t$ and describe how orthogonality is maintained.

\subsubsection{Right-Orthogonalization}
We perform right-orthogonalization on the cores $\{\mathcal{G}_k[t-1]\}_{k=1}^{n-1}$ and $\bm{g}_n[t-1]$ in the reverse order $k=n, n-1, \ldots, 2$.
Specifically, we first normalize $\bm{g}_n[t-1]$ and absorb its norm into the left-neighboring core $\mathcal{G}_{n-1}[t-1]$:
\begin{align*}
    \mathcal{G}_{n-1}[t-1] &\leftarrow \|\bm{g}_n[t-1]\|_2 \cdot \mathcal{G}_{n-1}[t-1], \\
    \bm{g}_n[t-1] &\leftarrow \frac{\bm{g}_n[t-1]}{\|\bm{g}_n[t-1]\|_2}.
\end{align*}
Next, we matricize the updated $\mathcal{G}_{n-1}[t-1]$ into $\mathbf{G}_{n-1}[t-1] \in \mathbb{R}^{r_{n-2} \times I_{n-1} r_{n-1}}$ and perform an LQ decomposition:
\begin{equation*}
    \mathbf{G}_{n-1}[t-1] = \mathbf{L} \mathbf{Q},
\end{equation*}
where $\mathbf{L} \in \mathbb{R}^{r_{n-2} \times r_{n-2}}$ is a lower triangular matrix, and $\mathbf{Q} \in \mathbb{R}^{r_{n-2} \times I_{n-1} r_{n-1}}$ satisfies $\mathbf{Q} \mathbf{Q}^\top = \mathbf{I}_{r_{n-2}}$.
Note that the LQ decomposition can be seen as the QR decomposition of the transpose of a matrix \citep{Matrix-Computations}.
Subsequently, we reshape the matrix $\mathbf{Q}$ back into a tensor of size $r_{n-2} \times I_{n-1} \times r_{n-1}$ to update $\mathcal{G}_{n-1}[t-1]$, and absorb $\mathbf{L}$ into the left-neighboring core $\mathcal{G}_{n-2}[t-1]$:
\begin{equation*}
    \mathcal{G}_{n-2}[t-1] \leftarrow \mathcal{G}_{n-2}[t-1] \ttimes{3}{1} \mathbf{L}.
\end{equation*}
This operation is repeated for $\mathcal{G}_{n-2}[t-1], \mathcal{G}_{n-3}[t-1], \ldots, \mathcal{G}_2[t-1]$, to complete the right-orthogonalization process (Figure~\ref{fig:tensor-diagram-right-ortho}).

\begin{algorithm}[!t]
\caption{Online TT-ALS}
\label{alg:online-tt-als}
\begin{algorithmic}[1]
\REQUIRE Data stream $\{\overline{\mathcal{X}}_t\}_{t=1}^T$, TT-ranks $\mathbf{r}$
\ENSURE Sequence of TT-decompositions $\{\{\mathcal{G}_k[t]\}_{k=1}^{n-1}, \bm{g}_n[t]\}_{t=1}^T$

\STATE Initialize orthogonal cores $\{\mathcal{G}_k\}_{k=1}^{n-1}$ and $\bm{g}_n$ randomly
\FOR{$t = 1$ to $T$}
    \STATE Observe current data slice $\overline{\mathcal{X}}_t$
    
    \STATE \textit{// 1. Right Orthogonalization}
    \STATE Normalize $\bm{g}_n$ and absorb norm into $\mathcal{G}_{n-1}$
    \FOR{$k = n-1$ down to $2$}
        \STATE $[\mathbf{L}, \mathbf{Q}] \leftarrow \text{LQ}(\mathcal{G}_k)$
        \STATE $\mathcal{G}_k \leftarrow \text{Fold}(\mathbf{Q})$
        \STATE $\mathcal{G}_{k-1} \leftarrow \mathcal{G}_{k-1} \ttimes{3}{1} \mathbf{L}$
    \ENDFOR

    \STATE \textit{// 2. Update Interfaces (Right side)}
    \STATE $\mathbf{B}_{n-1} \leftarrow \bm{g}_n$
    \FOR{$k = n-2$ down to $1$}
        \STATE $\mathbf{B}_k \leftarrow \mathcal{G}_{k+1} \ttimes{3}{1} \mathbf{B}_{k+1}$
    \ENDFOR

    \STATE \textit{// 3. Forward Update \& Left Orthogonalization}
    \STATE Initialize left interface $\mathbf{A}_{1} \leftarrow 1$
    \FOR{$k = 1$ to $n-1$}
        \STATE $\tilde{\mathcal{G}}_{k} \leftarrow (\mathbf{A}_{k})^\top \times \overline{\mathcal{X}}_t \times (\mathbf{B}_{k})^\top$ \quad \textit{// Local Solve}
        
        \STATE $[\mathbf{Q}, \mathbf{R}] \leftarrow \text{QR}(\tilde{\mathcal{G}}_{k})$
        \STATE $\mathcal{G}_{k} \leftarrow \text{Fold}(\mathbf{Q})$ \quad \textit{// Update Core}
        
        \IF{$k < n-1$}
            \STATE $\mathcal{G}_{k+1} \leftarrow \mathbf{R} \ttimes{2}{1} \mathcal{G}_{k+1}$ \quad \textit{// Propagate Error}
        \ELSE
            \STATE $\bm{g}_{n} \leftarrow \mathbf{R} \bm{g}_{n}$
        \ENDIF
        
        \STATE Update $\mathbf{A}_{k+1} \leftarrow \mathbf{A}_{k} \ttimes{k}{1} \mathcal{G}_{k}$ \quad \textit{// Update Left Interface}
    \ENDFOR
    
    \STATE $\bm{g}_{n} \leftarrow (\mathbf{A}_{n})^\top \cdot \text{vec}(\overline{\mathcal{X}}_t)$ \quad \textit{// Update Weight}

    \STATE \textbf{Store/Output} current decomposition $\{\mathcal{G}_{k}\}_{k=1}^{n-1}$ and $\bm{g}_{n}$ as time-$t$ result
\ENDFOR
\end{algorithmic}
\end{algorithm}

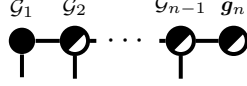
\begin{figure}[t]
  \centering
  \begin{tikzpicture}[node distance=0.7cm, on grid]
    \node[tensor, label={above: \scriptsize{$\mathcal{G}_1$}}] (G1) {};
    \node[orthogonal_tensor_r, right=of G1, label={above: \scriptsize{$\mathcal{G}_2$}}] (G2) {};
    \node[right=0.7cm of G2] (dots) {\large $\cdots$};
    \node[orthogonal_tensor_r, right=0.7cm of dots, label={above: \scriptsize{$\mathcal{G}_{n-1}$}}] (Gn-1) {};
    \node[orthogonal_tensor_r, right=of Gn-1, label={above: \scriptsize{$\bm{g}_n$}}] (Gn) {};
    \draw[bond] (G1) -- (G2) -- (dots) -- (Gn-1) -- (Gn);
    \foreach \n in {G1, G2, Gn-1} {
        \draw[bond] (\n.south) -- +(0,-0.3);
    }
  \end{tikzpicture}
  \caption{Tensor diagram after right-orthogonalization.}
  \label{fig:tensor-diagram-right-ortho}
\end{figure}

\subsubsection{Core Update and Left-Orthogonalization}
Next, we sequentially update each core in the order $k=1, 2, \ldots, n$ while simultaneously applying left-orthogonalization to the updated cores.
First, the contraction of the right-side cores, $\mathcal{B}_k[t-1]$, is precomputed for each $k=1, 2, \ldots, n$.
Then, starting with $k=1$, we update $\mathcal{G}_1[t]$ based on \eqref{eq:proposed-core-update}.
Since $\mathbf{A}_1[t] = 1$, the update is given by
\begin{equation*}
    \mathcal{G}_1[t] = \widehat{\mathcal{X}}_t \times \mathbf{B}_1[t-1]^\top.
\end{equation*}
Next, we matricize the updated $\mathcal{G}_1[t]$ into $\mathbf{G}_1[t] \in \mathbb{R}^{r_0 I_1 \times r_1}$ and perform a QR decomposition:
\begin{equation*}
    \mathbf{G}_1[t] = \mathbf{Q} \mathbf{R},
\end{equation*}
where $\mathbf{Q} \in \mathbb{R}^{r_0 I_1 \times r_1}$ satisfies $\mathbf{Q}^\top \mathbf{Q} = \mathbf{I}_{r_0 I_1}$, and $\mathbf{R} \in \mathbb{R}^{r_1 \times r_1}$ is upper triangular.
The matrix $\mathbf{Q}$ is then reshaped into a tensor of size $r_0 \times I_1 \times r_1$ to update $\mathcal{G}_1[t]$, and $\mathbf{R}$ is absorbed into the right-neighboring core $\mathcal{G}_2[t]$:
\begin{equation*}
    \mathcal{G}_2[t] \leftarrow \mathbf{R} \ttimes{2}{1} \mathcal{G}_2[t-1].
\end{equation*}
Using the resulting $\mathcal{G}_1[t]$, we then compute $\mathcal{A}_2[t]$ as follows:
\begin{equation*}
    \mathcal{A}_2[t] = \mathcal{A}_1[t] \times \mathcal{G}_1[t] = \mathcal{G}_1[t].
\end{equation*}
This procedure is repeated for $\mathcal{G}_2[t], \mathcal{G}_3[t], \ldots, \mathcal{G}_{n-1}[t]$, as illustrated in Figure~\ref{fig:tensor-diagram-updating}.
Finally, we update $\bm{g}_n[t]$ based on \eqref{eq:proposed-core-update}:
\begin{equation*}
    \bm{g}_n[t] = \mathbf{A}_n[t]^\top \times \widehat{\mathcal{X}}_t.
\end{equation*}
Note that here, $\mathbf{B}_n[t-1]$ is a scalar equal to 1.
This completes the update procedure at time $t$.

\begin{figure}[t]
  \centering
  \begin{tikzpicture}[node distance=0.7cm, on grid]
    \node[orthogonal_tensor_l, label={above: \scriptsize{$\mathcal{G}_1$}}] (G1) {};
    \node[orthogonal_tensor_l, right=of G1, label={above: \scriptsize{$\mathcal{G}_2$}}] (G2) {};
    \node[right=0.7cm of G2] (dots1) {\large $\cdots$};
    \node[orthogonal_tensor_l, right=0.7cm of dots1, label={above: \scriptsize{$\mathcal{G}_{k-1}$}}] (Gk-1) {};
    \node[tensor, right=of Gk-1, label={above: \scriptsize{$\mathcal{G}_{k}$}}] (Gk) {};
    \node[orthogonal_tensor_r, right=of Gk, label={above: \scriptsize{$\mathcal{G}_{k+1}$}}] (Gk+1) {};
    \node[right=0.7cm of Gk+1] (dots2) {\large $\cdots$};
    \node[orthogonal_tensor_r, right=0.7cm of dots2, label={above: \scriptsize{$\mathcal{G}_{n-2}$}}] (Gn-2) {};
    \node[orthogonal_tensor_r, right=of Gn-2, label={above: \scriptsize{$\mathcal{G}_{n-1}$}}] (Gn-1) {};
    \node[orthogonal_tensor_r, right=of Gn-1, label={above: \scriptsize{$\bm{g}_n$}}] (Gn) {};
    \draw[bond] (G1) -- (G2) -- (dots1) -- (Gk-1) -- (Gk) -- (Gk+1) -- (dots2) -- (Gn-2) -- (Gn-1) -- (Gn);
    \foreach \n in {G1, G2, Gk-1, Gk, Gk+1, Gn-2, Gn-1} {
        \draw[bond] (\n.south) -- +(0,-0.3);
    }
  \end{tikzpicture}
  \caption{Tensor diagram during the update of $\mathcal{G}_k$.}
  \label{fig:tensor-diagram-updating}
\end{figure}
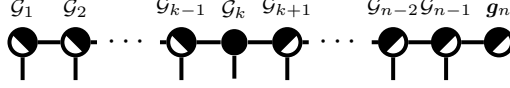

\section{Theoretical and Computational Analysis}
\label{sec:theoretical_and_computational_analysis}
This section analyzes the proposed Online TT-ALS in terms of theoretical stability and computational efficiency.
By rigorously enforcing orthogonal gauge constraints, our method simultaneously guarantees monotonic decrease of the local objective function (Section~\ref{subsec:theoretical_analysis}) and significantly reduces computational complexity without randomized approximations (Section~\ref{subsec:computational_complexity}).

\subsection{Theoretical Analysis}
\label{subsec:theoretical_analysis}
We establish the theoretical foundation of the proposed method.
Let $f(\mathcal{G}_k)$ be the objective function defined in \eqref{eq:proposed-objective-matrix}.
To analyze the local optimization process, let $\mathcal{G}_k^{(\text{old})}$ denote the state of the $k$-th core tensor prior to the update, and let $\mathcal{G}_k^\star$ denote the newly updated core tensor obtained via \eqref{eq:proposed-core-update}.

\begin{proposition}[Monotonicity]
    \label{prop:monotonic_convergence}
    Under the orthogonality constraints $\mathbf{A}_k[t]^\top \mathbf{A}_k[t] = \mathbf{I}_{r_{k-1}}$ and $\mathbf{B}_k[t-1] \mathbf{B}_k[t-1]^\top = \mathbf{I}_{r_k}$, the update from $\mathcal{G}_k^{(\text{old})}$ to $\mathcal{G}_k^\star$ ensures $f(\mathcal{G}_k^\star) \leq f(\mathcal{G}_k^{(\text{old})})$.
\end{proposition}
\begin{proof}
    Let $\mathcal{G}_k^\star$ be the updated core tensor, which satisfies $\frac{\partial f}{\partial \mathcal{G}_k} \bigg|_{\mathcal{G}_k = \mathcal{G}_k^\star} = 0$.
    For an arbitrary core tensor $\mathcal{G}_k$, we can express it as $\mathcal{G}_k = \mathcal{G}_k^\star + \Delta\mathcal{G}$.
    We define the residual as $\mathcal{R} = \widehat{\mathcal{X}}_t - \mathbf{A}_k[t] \ttimes{k}{1} \mathcal{G}_k^\star \ttimes{k+1}{1} \mathbf{B}_k[t-1]$, and the propagated difference as $\mathcal{D} = \mathbf{A}_k[t] \ttimes{k}{1} \Delta\mathcal{G} \ttimes{k+1}{1} \mathbf{B}_k[t-1]$.
    Substituting these into \eqref{eq:proposed-objective-matrix}, the objective function is $f(\mathcal{G}_k) = \frac{1}{2} \|\mathcal{R} - \mathcal{D}\|_\mathrm{F}^2$.
    Expanding this expression yields:
    \begin{align*}
        f(\mathcal{G}_k) = f(\mathcal{G}_k^\star) - \sum_{l,m,n} \mathcal{R}(l,m,n)\mathcal{D}(l,m,n) + \frac{1}{2} \|\mathcal{D}\|_\mathrm{F}^2.
    \end{align*}
    Expanding mode products, the second term becomes:
    \begin{align*}
        \sum_{l,m,n} \mathcal{R}(l,m,n)\mathcal{D}(l,m,n) &= \sum_{l,m,n} \mathcal{R}(l,m,n) \left( \sum_{\alpha, \beta} \mathbf{A}_k[t](l,\alpha) \Delta\mathcal{G}(\alpha,m,\beta) \mathbf{B}_k[t-1](\beta,n) \right) \\
        &= \sum_{\alpha, m, \beta} \Delta\mathcal{G}(\alpha,m,\beta) \left( \sum_{l,n} (\mathbf{A}_k[t])^\top(\alpha,l) \mathcal{R}(l,m,n) (\mathbf{B}_k[t-1])^\top(n,\beta) \right)
    \end{align*}
    The inner summation equals $\mathbf{A}_k[t]^\top \times \mathcal{R} \times \mathbf{B}_k[t-1]^\top$.
    Applying the orthogonality, we have
    \begin{align*}
        \mathbf{A}_k[t]^\top \times \mathcal{R} \times \mathbf{B}_k[t-1]^\top &= \mathbf{A}_k[t]^\top \times \left( \widehat{\mathcal{X}}_t - \mathbf{A}_k[t] \times \mathcal{G}_k^\star \times \mathbf{B}_k[t-1] \right) \times \mathbf{B}_k[t-1]^\top \\
        &= \mathbf{A}_k[t]^\top \times \widehat{\mathcal{X}}_t \times \mathbf{B}_k[t-1]^\top - \mathcal{G}_k^\star \\
        &= -\frac{\partial f}{\partial \mathcal{G}_k} \bigg|_{\mathcal{G}_k = \mathcal{G}_k^\star} = 0.
    \end{align*}
    Thus, the second term reduces to zero.
    The orthogonality also indicates that $\|\mathcal{D}\|_\mathrm{F}^2 = \|\Delta\mathcal{G}\|_\mathrm{F}^2$.
    Therefore, we obtain $f(\mathcal{G}_k) = f(\mathcal{G}_k^\star) + \frac{1}{2} \|\Delta\mathcal{G}\|_\mathrm{F}^2 \geq f(\mathcal{G}_k^\star)$, which completes the proof.
\end{proof}

\begin{proposition}[Temporal Smoothness Bound]
    Let $\mathcal{M}_{t-1}$ be the reconstructed tensor slice at time $t-1$ with an approximation error $E_{t-1} = \|\widehat{\mathcal{X}}_{t-1} - \mathcal{M}_{t-1}\|_\mathrm{F}$.
    If the variation in the streaming data is bounded by $\|\widehat{\mathcal{X}}_t - \widehat{\mathcal{X}}_{t-1}\|_\mathrm{F} \leq \epsilon$, then the final error $E_t$ evaluated on the new data is bounded by $E_t \leq E_{t-1} + \epsilon$.     
\end{proposition}
\begin{proof}
    Right-orthogonalization alters the gauge but preserves $\mathcal{M}_{t-1}$.
    Thus, the initial error evaluated on the new data $\widehat{\mathcal{X}}_t$ is $\|\widehat{\mathcal{X}}_t - \mathcal{M}_{t-1}\|_\mathrm{F}$.
    By the triangle inequality, we have $\|\widehat{\mathcal{X}}_t - \mathcal{M}_{t-1}\|_\mathrm{F} \leq \|\widehat{\mathcal{X}}_t - \widehat{\mathcal{X}}_{t-1}\|_\mathrm{F} + \|\widehat{\mathcal{X}}_{t-1} - \mathcal{M}_{t-1}\|_\mathrm{F} \leq \epsilon + E_{t-1}$.
    Since Proposition~\ref{prop:monotonic_convergence} guarantees a monotonic decrease, $E_t \leq \|\widehat{\mathcal{X}}_t - \mathcal{M}_{t-1}\|_\mathrm{F} \leq E_{t-1} + \epsilon$.
\end{proof}

\subsection{Computational Analysis}
\label{subsec:computational_complexity}
While the strict enforcement of orthogonal constraints guarantees theoretical stability, it might initially appear that calculating these exact algebraic updates and tracking the gauges via QR/LQ decompositions at every time step would introduce a computational bottleneck.
However, the rigorously maintained orthogonal subspace decouples the optimization, allowing us to substantially reduce the dimensionality of the intermediate tensor contractions.

In the following, we provide a detailed analysis of the computational complexity of the proposed method per time step for an $n$-th order streaming tensor.
We assume that the size of each mode is $I$ and the TT-ranks are uniform ($r = r_1 = \cdots = r_{n-1}$), with $I \gg r \gg n$.
Tensor matricization and reshaping are excluded from the analysis as they are $\mathcal{O}(1)$ operations via index manipulation.

\paragraph{Analysis of Right Orthogonalization}
In this step, we process $n-1$ TT cores sequentially from right to left. The processing of each TT core consists of two stages:
\begin{itemize}
    \item \textbf{LQ Decomposition:} An input matrix of size $r \times Ir$ is decomposed using LQ decomposition. The computational cost is $\mathcal{O}(I r^3)$.
    \item \textbf{Absorption of $\mathbf{L}$:} The resulting lower triangular matrix $\mathbf{L}$ (size $r \times r$) is multiplied into the left adjacent core (size $rI \times r$). The computational cost is $\mathcal{O}(I r^3)$.
\end{itemize}
Since these operations are performed for $n-1$ TT cores, the total complexity of the right orthogonalization phase is $\mathcal{O}(n I r^3)$.

\paragraph{Analysis of Contraction Computations}
We consider the computation of $\mathcal{A}_k[t]$ and $\mathcal{B}_k[t-1]$ required for updating the TT cores.
\begin{itemize}
    \item \textbf{Computation of $\mathcal{A}_k[t]$:} This involves the matrix multiplication of a matrix of size $I^{k-2} \times r$ and a matrix of size $r \times Ir$. The cost is $\mathcal{O}(I^{k-1} r^2)$. Since the sum over $k=1, \ldots, n-1$ is dominated by the term at $k=n-1$, the complexity is $\mathcal{O}(I^{n-2} r^2)$.
    \item \textbf{Computation of $\mathcal{B}_k[t-1]$:} This involves the matrix multiplication of a matrix of size $rI \times r$ and a matrix of size $r \times (I^{n-k-2})$. The cost is $\mathcal{O}(I^{n-k-1} r^2)$. Since the sum over $k=1, \ldots, n-1$ is dominated by the term at $k=1$, the complexity is $\mathcal{O}(I^{n-2} r^2)$.
\end{itemize}
Combining these, the total complexity for computing the left and right contractions is $\mathcal{O}(I^{n-2} r^2)$.

\paragraph{Analysis of Core Update}
The update of each TT core based on \eqref{eq:proposed-core-update} involves contraction calculations of three components: a matrix of size $r \times I^{k-1}$, a tensor of size $I^{k-1} \times I \times I^{n-k-1}$, and a matrix of size $(I^{n-k-1} \times r)$.
\begin{itemize}
    \item \textbf{Case $k=1$:} Since $\mathcal{A}_1[t] = 1$, the complexity is $\mathcal{O}(I^{n-1} r)$.
    \item \textbf{Case $2 \leq k \leq n-2$:} The complexity is determined by the sum of contraction costs, resulting in $\mathcal{O}(I^{n-1} r + I^{n-k} r^2)$.
    \item \textbf{Case $k=n-1$:} Since $\mathcal{B}_n[t-1] = 1$, the complexity is $\mathcal{O}(I^{n-1} r)$.
\end{itemize}
Therefore, the total computational cost for updating all TT cores is $\mathcal{O}(I^{n-1} r + I^{n-2} r^2)$.

\paragraph{Analysis of Left Orthogonalization}
For each updated TT core, the following two stages are performed:
\begin{itemize}
    \item \textbf{QR Decomposition:} A matrix of size $I r \times r$ is decomposed using QR decomposition. The computational cost is $\mathcal{O}(I r^3)$.
    \item \textbf{Absorption of $\mathbf{R}$:} The resulting upper triangular matrix $\mathbf{R}$ (size $r \times r$) is multiplied into the right adjacent core (size $r \times Ir$). The computational cost is $\mathcal{O}(I r^3)$.
\end{itemize}
Since these operations are performed for $n-1$ TT cores, the total complexity of the left orthogonalization phase is $\mathcal{O}(n I r^3)$.

\paragraph{Total Complexity and Discussion}
Based on the above analysis, the total computational complexity of the proposed method at time $t$ is $\mathcal{O}(I^{n-1} r + I^{n-2} r^2 + n I r^3)$. Under the practical assumption $I \gg r \gg n$, the dominant term becomes $\mathcal{O}(I^{n-1} r)$.

This indicates that the computational cost of our exact single-sweep update scales linearly with the raw data size of the incoming tensor slice ($\prod_{i=1}^{n-1} I = I^{n-1}$).
Furthermore, this $\mathcal{O}(I^{n-1} r)$ complexity represents a significant theoretical advantage over the preceding TT-FOA method, which requires $\mathcal{O}(I^{n-1} r^2)$.
By scaling linearly rather than quadratically with respect to the TT-rank $r$, our deterministic approach achieves millisecond-level processing times without compromising tracking stability or relying on randomized approximations.

\section{Experiments}
\label{sec:experiments}
In this section, we report the results of comprehensive experiments to demonstrate the effectiveness of the proposed Online TT-ALS method.
The primary objectives of these evaluations are fourfold: (i) to verify the scalability of our method on high-order streaming tensors, (ii) to evaluate its numerical convergence and perceptual quality on real-world videos with both static and dynamic backgrounds, (iii) to compare its computational latency against modern deep learning-based paradigms, and (iv) to validate the necessity of the proposed orthogonality for long-term tracking stability.

\subsection{Experimental Setup and Evaluation Metrics}
We evaluated our method using both synthetic high-order tensors and real-world video datasets:
\begin{itemize}
    \item \textbf{Synthetic High-Order Tensors:} To test scalability, we generated streaming tensors of order $n=5, 6, 7$ with a spatial dimension of $I = 30$ per mode and $T = 100$ time steps. To mimic real-world data properties, these tensors were constructed by contracting randomly generated TT cores whose matricized forms exhibited decaying singular values.
    \item \textbf{Static Background Videos:} We selected the \textit{office} sequence as a grayscale video (third-order, $240 \times 360 \times 500$) and the \textit{pedestrians} sequence as a color video (fourth-order, $240 \times 360 \times 3 \times 500$) from the ``baseline'' category of the CDnet 2014 benchmark \citep{CDNet2014}\footnote{\url{http://www.changedetection.net/}}. Both videos feature a static background with moving people.
    \item \textbf{Dynamic Background Videos:} To test tracking robustness against camera motion, we utilized the \textit{continuousPan} sequence from the ``PTZ'' category of the CDnet 2014 benchmark.
\end{itemize}

For comparative evaluation, we assessed the proposed method against several algebraic batch and online methods, as well as a recent neural network-based approach. The evaluated algorithms are as follows:
\begin{itemize}
    \item \textbf{TT-ALS} \citep{TT-ALS}: A standard batch-based TT decomposition algorithm.
    \item \textbf{TT-ALS (Slice)} \citep{TT-ALS}: A method in which TT-ALS is applied sequentially to the data slice at each time step.
    \item \textbf{TT-ALS (Batch)} \citep{TT-ALS}: A variant in which TT-ALS processes frames in small sliding windows (batch size $= 10$).
    \item \textbf{TT-FOA} \citep{TT-FOA}: An existing online TT decomposition algorithm based on first-order approximation.
    \item \textbf{OFTD} \citep{OFTD}: A recent deep learning-based continuous functional approximation framework.
    \item \textbf{Online TT-ALS (Ours)}: The exact, single-sweep online TT decomposition algorithm proposed in this paper.
\end{itemize}
All algorithms were implemented in Julia (version 1.11.5) and evaluated on a workstation equipped with an AMD EPYC 7702P 64-Core Processor and 512 GB of RAM, running Ubuntu 20.04.6 LTS.
For a fair comparison, all computations were executed using a single thread, except for OFTD, which was evaluated using its official Python implementation.

In our experiments, we reconstructed the original tensors from the obtained Tensor Train decompositions and evaluated the reconstruction error.
However, particularly for the video datasets, relying solely on relative error may not adequately reflect the perceptual quality.

\begin{figure}[t]
    \centering
    \begin{subfigure}{0.32\linewidth}
        \includegraphics[width=\linewidth]{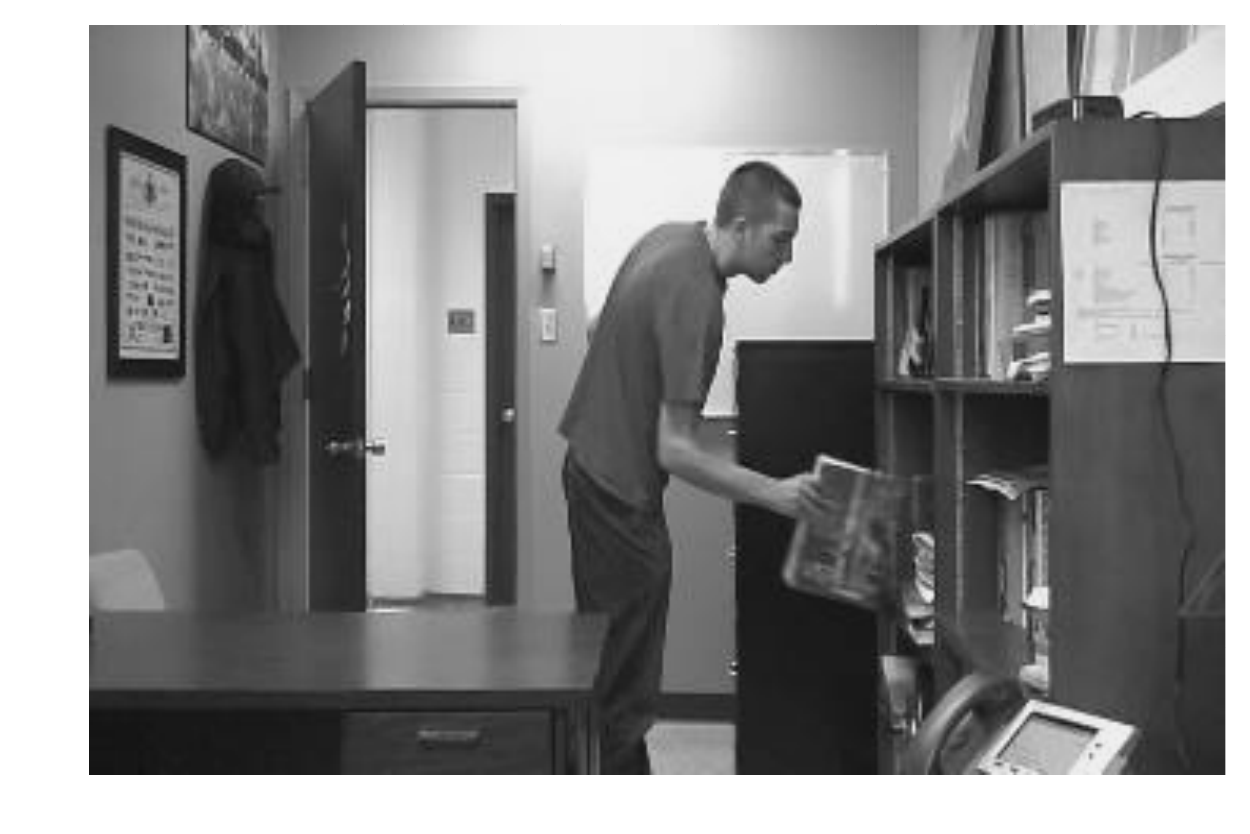}
        \caption{Original Frame}
    \end{subfigure}
    \begin{subfigure}{0.32\linewidth}
        \includegraphics[width=\linewidth]{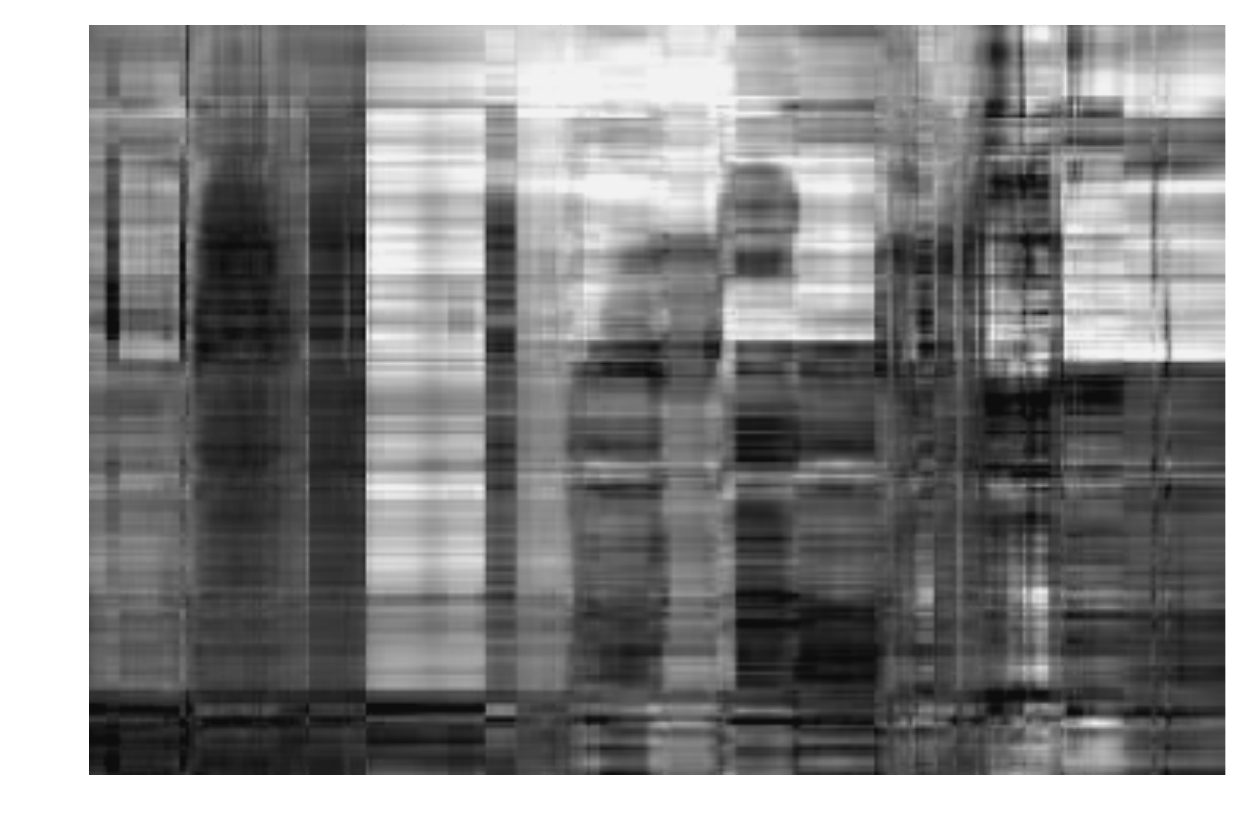}
        \caption{TT-ALS (Slice)}
    \end{subfigure}
    \begin{subfigure}{0.32\linewidth}
        \includegraphics[width=\linewidth]{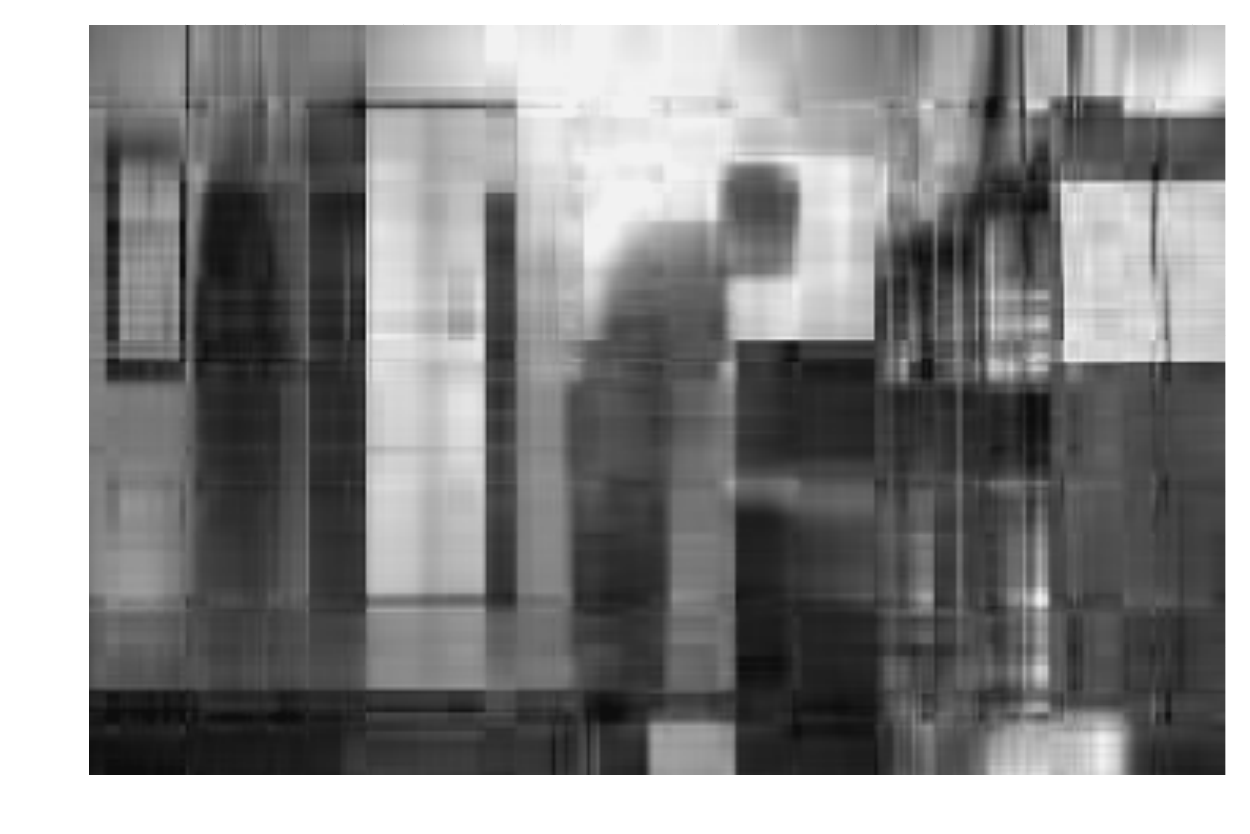}
        \caption{TT-FOA}
    \end{subfigure}
    \caption{Comparison of reconstructed images at frame 350 of the grayscale video ($r=(10,10)$). 
    Although the RE values are comparable ($\approx 0.20$), TT-FOA yields a clearer reconstruction of the person with less noise.
    }
    \label{fig:visual_comparison_focus}
\end{figure}
For instance, as shown later in Table \ref{tab:comparison_gray_10}, the Relative Error (RE) for TT-ALS (Slice) is 0.205, which is comparable to 0.206 for TT-FOA.
However, a visual comparison of the reconstructed images in Figure~\ref{fig:visual_comparison_focus} reveals that TT-FOA produces a clearer image with less noise.
This demonstrates that a small numerical error does not necessarily guarantee high perceptual quality.
To address this issue, we employed the following seven metrics for evaluation, encompassing both mathematical error and perceptual quality:
\begin{itemize}
    \item \textbf{RE} (Relative Error): A metric measuring the numerical approximation error of the entire tensor defined by the Frobenius norm.
    \item \textbf{M-RMSE} (Masked Root Mean Square Error): RMSE calculated specifically for the moving object (foreground) regions extracted via background subtraction.
    \item \textbf{PSNR} (Peak Signal-to-Noise Ratio) \citep{Digital-Image-Processing}: A standard quality metric in video signal processing based on the pixel-wise mean squared error.
    \item \textbf{SSIM} (Structural Similarity) \citep{SSIM}: A metric based on human visual characteristics, integrating luminance, contrast, and structural information.
    \item \textbf{MS-SSIM} (Multi-Scale SSIM) \citep{MS-SSIM}: An extension of SSIM evaluated across multiple resolution scales to enhance stability.
    \item \textbf{LPIPS} (Learned Perceptual Image Patch Similarity) \citep{LPIPS}: A perceptual similarity metric based on the distance between feature maps in deep neural networks.
    \item \textbf{VMAF} (Video Multi-method Assessment Fusion) \citep{VMAF}: An industry-standard video quality metric trained to predict human subjective quality ratings.
\end{itemize}
We categorized these seven metrics into three groups for our analysis.
The first category comprises Mathematical Metrics, which are fundamental metrics based on numerical pixel-wise errors.
This category includes RE for overall approximation accuracy, M-RMSE for focused evaluation of moving object regions, and PSNR for the signal-to-noise ratio.
The second category is Structural Metrics.
Even if the numerical error is low, structural details such as edges may be lost, resulting in blur artifacts.
Therefore, considering human visual characteristics, we adopted SSIM and MS-SSIM to assess the preservation of luminance, contrast, and structural information.
The final category is Perceptual and Video Quality Metrics.
To more rigorously evaluate the results in terms of human visual perception, beyond purely numerical reconstruction accuracy, we employed LPIPS, which leverages deep learning feature representations, as well as VMAF, an industry-standard metric for video quality assessment.
These metrics quantify the perceived visual quality of the reconstructed video.

In addition to quality metrics, we measured the average execution time required per frame for each algorithm.
Note that for the batch-based TT-ALS, the total time required to process the entire sequence was recorded as the average execution time.

\subsection{Scalability on High-Order Tensors}
\begin{table}[!t]
    \centering
    \caption{Scalability on higher-order streaming tensors. Results are shown as (Relative Error / Average Time per Frame in ms). ``OOM'' indicates Out-Of-Memory.}
    \label{tab:high_order_scalability}
        \begin{tabular}{l c c c}
            \toprule
            \textbf{Method} & $\mathbf{n=5}$ & $\mathbf{n=6}$ & $\mathbf{n=7}$ \\
            \midrule
            TT-ALS      & 0.306 / 56.91 & 0.240 / 1442.0 & N/A (OOM) \\
            TT-ALS (Slice)    & 0.253 / 32.95 & 0.265 / 1338.0 & 0.269 / 47253.0 \\
            TT-ALS (Batch) & 0.276 / 56.14 & 0.260 / 1769.0 & N/A (OOM) \\
            TT-FOA            & \textbf{0.141} / 81.43 & \textbf{0.165} / 2573.0 & 0.609 / 78043.0 \\
            Online TT-ALS (Ours) & 0.166 / \textbf{15.21} & 0.178 / \textbf{545.5} & \textbf{0.177} / \textbf{17365.0} \\
            \bottomrule
        \end{tabular}
\end{table}

The Tensor Train format is fundamentally designed to avoid the curse of dimensionality, a theoretical bottleneck where the number of parameters scales exponentially with the tensor order in traditional frameworks such as Tucker decomposition.
While the TT format theoretically offers linear scaling, maintaining this efficiency and computational stability in an online streaming setting remains a significant challenge.

To demonstrate that our proposed Online TT-ALS achieves this high-dimensional scalability in practice, we evaluated the algorithms on synthetic higher-order streaming tensors ($n=5, 6, 7$).
This evaluation explicitly shows how our method overcomes the memory exhaustion and computational delays inherent in conventional batch and existing online TT methods.
As shown in Table \ref{tab:high_order_scalability}, both the full-batch and mini-batch TT-ALS failed due to memory exhaustion (Out-Of-Memory, OOM) at $n=7$.
TT-FOA successfully processed the data but required prohibitively high computational cost (over $78$ seconds per frame) with significant error degradation.
In contrast, our proposed method maintained stable approximation accuracy and scaled efficiently.
This result numerically verifies the theoretical $\mathcal{O}(I^{n-1} r)$ linear complexity of the proposed algorithm, proving its practical applicability for high-dimensional streaming environments.

\subsection{Performance on Streaming Video Data}
\subsubsection{Accuracy and Perceptual Quality}
We conducted extensive tracking and compression experiments on real-world video sequences featuring both static and dynamic backgrounds. In the quantitative evaluation tables presented below, the best results are highlighted in \textbf{bold}, and the second-best results are \underline{underlined}. The arrows indicate whether lower ($\downarrow$) or higher ($\uparrow$) values denote better performance.

\paragraph{Quantitative Evaluation}
For the static background scenarios, as summarized in Tables \ref{tab:comparison_gray_10} through \ref{tab:comparison_color_30}, the proposed Online TT-ALS demonstrated superior performance.
Overall, it outperformed other methods not only in mathematical metrics such as RE and PSNR but also in structural and perceptual metrics such as SSIM and VMAF.
A detailed analysis reveals that while the batch-based methods (TT-ALS and TT-ALS (Slice)) achieved mathematical metrics comparable to the online methods (TT-FOA and the proposed method), they yielded inferior results in terms of structural and perceptual metrics.
This disparity is particularly pronounced in the VMAF scores, where a significant gap exists between batch and online approaches. 

Furthermore, the proposed method significantly reduced mathematical errors even under strict low-rank constraints (e.g., $r=(10,10)$ and $r=(10,3,3)$).
Consequently, metrics based on human visual perception, such as SSIM and VMAF, were substantially improved.
Even with higher ranks, our method consistently maintained the best performance, with particularly notable gains in VMAF.
Regarding computational efficiency, although TT-ALS (Slice) exhibited shorter average execution times in certain settings, the proposed method was consistently faster than TT-FOA, demonstrating its capability for practical online processing.

This quantitative advantage extends to scenarios with dynamic camera motion. As shown in Table \ref{tab:comparison_dynamic_background} for the \textit{continuousPan} sequence under a high-rank setting ($r=(30,3,3)$), our method consistently outperformed all baselines across every evaluation metric, proving its robustness against complex background variations.

\begin{table}[t]
    \centering
    \begin{minipage}[t]{0.48\linewidth}
        \centering
        \caption{Quantitative comparison of algorithms on the grayscale video ($r=(10,10)$).}
        \label{tab:comparison_gray_10}
        \resizebox{\linewidth}{!}{
        \setlength{\tabcolsep}{2pt}
        \begin{tabular}{l cc c cc c cc c}
        \toprule
        & \multicolumn{3}{c}{\textbf{Mathematical}} & \multicolumn{2}{c}{\textbf{Structural}} & \multicolumn{2}{c}{\textbf{Perceptual}} & \textbf{Time} \\
        \cmidrule(lr){2-4} \cmidrule(lr){5-6} \cmidrule(lr){7-8} \cmidrule(lr){9-9}
        \textbf{Method} & RE $\downarrow$ & M-RMSE $\downarrow$ & PSNR $\uparrow$ & SSIM $\uparrow$ & MS-SSIM $\uparrow$ & LPIPS $\downarrow$ & VMAF $\uparrow$ & \makecell{Mean Time \\ (ms) $\downarrow$} \\
        \midrule
        TT-ALS                & \underline{0.193} & 0.184 & 20.50 & 0.549 & 0.749 & 0.541 & 19.68 & 2138.088 \\
        TT-ALS (Slice)        & 0.205 & 0.161 & 20.04 & 0.534 & 0.755 & 0.523 & 21.20 & \textbf{2.244} \\
        TT-FOA                & 0.206 & \underline{0.141} & \underline{23.30} & \underline{0.740} & \underline{0.875} & \underline{0.369} & \underline{39.85} & 5.370 \\
        Online TT-ALS (Ours) & \textbf{0.128} & \textbf{0.121} & \textbf{24.02} & \textbf{0.766} & \textbf{0.902} & \textbf{0.346} & \textbf{42.35} & \underline{2.756} \\
        \bottomrule
        \end{tabular}
        }
    \end{minipage}
    \hfill
    \begin{minipage}[t]{0.48\linewidth}
        \centering
        \caption{Quantitative comparison of algorithms on the grayscale video ($r=(30,30)$).}
        \label{tab:comparison_gray_30}
        \resizebox{\linewidth}{!}{
        \setlength{\tabcolsep}{2pt}
        \begin{tabular}{l cc c cc c cc c}
        \toprule
        & \multicolumn{3}{c}{\textbf{Mathematical}} & \multicolumn{2}{c}{\textbf{Structural}} & \multicolumn{2}{c}{\textbf{Perceptual}} & \textbf{Time} \\
        \cmidrule(lr){2-4} \cmidrule(lr){5-6} \cmidrule(lr){7-8} \cmidrule(lr){9-9}
        \textbf{Method} & RE $\downarrow$ & M-RMSE $\downarrow$ & PSNR $\uparrow$ & SSIM $\uparrow$ & MS-SSIM $\uparrow$ & LPIPS $\downarrow$ & VMAF $\uparrow$ & \makecell{Mean Time \\ (ms) $\downarrow$} \\
        \midrule
        TT-ALS                & 0.124 & 0.120 & 24.37 & 0.669 & 0.880 & 0.397 & 47.18 & 2307.555 \\
        TT-ALS (Slice)        & \underline{0.108} & \underline{0.090} & 25.53 & 0.704 & 0.907 & 0.357 & 53.73 & \textbf{3.963} \\
        TT-FOA                & 0.171 & 0.102 & \underline{28.27} & \underline{0.841} & \underline{0.944} & \underline{0.211} & \underline{76.76} & 268.663 \\
        Online TT-ALS (Ours) & \textbf{0.068} & \textbf{0.075} & \textbf{29.54} & \textbf{0.874} & \textbf{0.973} & \textbf{0.173} & \textbf{82.92} & \underline{24.485} \\
        \bottomrule
        \end{tabular}
        }
    \end{minipage}
\end{table}

\begin{table*}[t]
    \centering
    \begin{minipage}[t]{0.48\linewidth}
        \centering
        \caption{Quantitative comparison of algorithms on the color video ($r=(10,3,3)$).}
        \label{tab:comparison_color_10}
        \resizebox{\linewidth}{!}{
        \setlength{\tabcolsep}{2pt}
        \begin{tabular}{l cc c cc c cc c}
        \toprule
        & \multicolumn{3}{c}{\textbf{Mathematical}} & \multicolumn{2}{c}{\textbf{Structural}} & \multicolumn{2}{c}{\textbf{Perceptual}} & \textbf{Time} \\
        \cmidrule(lr){2-4} \cmidrule(lr){5-6} \cmidrule(lr){7-8} \cmidrule(lr){9-9}
        \textbf{Method} & RE $\downarrow$ & M-RMSE $\downarrow$ & PSNR $\uparrow$ & SSIM $\uparrow$ & MS-SSIM $\uparrow$ & LPIPS $\downarrow$ & VMAF $\uparrow$ & \makecell{Mean Time \\ (ms) $\downarrow$} \\
        \midrule
        TT-ALS                & 0.161 & 0.341 & 17.92 & 0.322 & 0.436 & 0.575 & 3.79 & 6124.260 \\
        TT-ALS (Slice)        & \underline{0.113} & \underline{0.176} & 20.90 & 0.554 & 0.710 & 0.443 & 6.28 & 11.910 \\
        TT-FOA                & 0.169 & 0.198 & \underline{23.36} & \underline{0.696} & \underline{0.828} & \underline{0.389} & \underline{11.47} & \underline{11.328} \\
        Online TT-ALS (Ours) & \textbf{0.073} & \textbf{0.136} & \textbf{24.57} & \textbf{0.719} & \textbf{0.859} & \textbf{0.356} & \textbf{19.86} & \textbf{3.904} \\
        \bottomrule
        \end{tabular}
        }
    \end{minipage}
    \hfill
    \begin{minipage}[t]{0.48\linewidth}
        \centering
        \caption{Quantitative comparison of algorithms on the color video ($r=(30,3,3)$).}
        \label{tab:comparison_color_30}
        \resizebox{\linewidth}{!}{
        \setlength{\tabcolsep}{2pt}
        \begin{tabular}{l cc c cc c cc c}
        \toprule
        & \multicolumn{3}{c}{\textbf{Mathematical}} & \multicolumn{2}{c}{\textbf{Structural}} & \multicolumn{2}{c}{\textbf{Perceptual}} & \textbf{Time} \\
        \cmidrule(lr){2-4} \cmidrule(lr){5-6} \cmidrule(lr){7-8} \cmidrule(lr){9-9}
        \textbf{Method} & RE $\downarrow$ & M-RMSE $\downarrow$ & PSNR $\uparrow$ & SSIM $\uparrow$ & MS-SSIM $\uparrow$ & LPIPS $\downarrow$ & VMAF $\uparrow$ & \makecell{Mean Time \\ (ms) $\downarrow$} \\
        \midrule
        TT-ALS                & 0.095 & 0.337 & 22.58 & 0.647 & 0.779 & 0.336 & 11.96 & 6379.705 \\
        TT-ALS (Slice)        & \underline{0.066} & \underline{0.093} & 25.63 & 0.708 & 0.871 & 0.239 & 36.92 & 18.037 \\
        TT-FOA                & 0.166 & 0.142 & \underline{27.05} & \underline{0.814} & \underline{0.915} & \underline{0.211} & \underline{54.53} & \underline{14.782} \\
        Online TT-ALS (Ours) & \textbf{0.040} & \textbf{0.068} & \textbf{29.90} & \textbf{0.850} & \textbf{0.953} & \textbf{0.144} & \textbf{72.01} & \textbf{5.911} \\
        \bottomrule
        \end{tabular}
        }
    \end{minipage}
\end{table*}

\begin{table}[t]
\centering
\caption{Quantitative comparison of algorithms on the dynamic background video ($r=(30,3,3)$).}
\label{tab:comparison_dynamic_background}
\resizebox{\columnwidth}{!}{
\begin{tabular}{l cc c cc c cc c}
\toprule
& \multicolumn{3}{c}{\textbf{Mathematical}} & \multicolumn{2}{c}{\textbf{Structural}} & \multicolumn{2}{c}{\textbf{Perceptual}} & \textbf{Time} \\
\cmidrule(lr){2-4} \cmidrule(lr){5-6} \cmidrule(lr){7-8} \cmidrule(lr){9-9}
\textbf{Method} & RE $\downarrow$ & M-RMSE $\downarrow$ & PSNR $\uparrow$ & SSIM $\uparrow$ & MS-SSIM $\uparrow$ & LPIPS $\downarrow$ & VMAF $\uparrow$ & \makecell{Mean Time \\ (ms) $\downarrow$} \\
\midrule
TT-ALS                & 0.203 & 0.199 & 18.22 & 0.526 & 0.472 & 0.621 & 15.53 & 25177.4 \\
TT-ALS (Slice)        & \underline{0.112} & \underline{0.093} & \underline{23.62} & 0.601 & 0.737 & \underline{0.406} & 51.83 & 65.51 \\
TT-FOA                & 0.193 & 0.112 & 23.52 & \underline{0.702} & \underline{0.815} & 0.441 & \underline{53.49} & \underline{57.76} \\
Online TT-ALS (Ours) & \textbf{0.085} & \textbf{0.073} & \textbf{26.03} & \textbf{0.721} & \textbf{0.851} & \textbf{0.347} & \textbf{66.98} & \textbf{22.16} \\
\bottomrule
\end{tabular}
}
\end{table}

\paragraph{Qualitative Evaluation}
In addition to quantitative metrics, we visually compared the quality of the reconstructed images to assess their perceptual fidelity.
From Figures \ref{fig:visual_comparison_gray} and \ref{fig:visual_comparison_color}, we observe that TT-ALS and TT-ALS (Slice) generally exhibit blocking artifacts.
In particular, the contours of the person and the background textures are lost, resulting in a blurred appearance. Regarding TT-FOA, motion blur is evident in areas containing fine details, such as the fingertips of the moving person.

In contrast, the proposed method preserves the edges associated with the person's motion and reconstructs the background textures with high detail.
These visual results are consistent with the superior scores in structural and perceptual metrics, such as SSIM and LPIPS, reported earlier.
This confirms that the proposed method achieves high-quality video compression and reconstruction that is well-aligned with human visual characteristics.

\begin{figure*}[t]
    \centering
    \begin{minipage}{0.19\linewidth} \centering \small \textbf{Original} \end{minipage}
    \hfill
    \begin{minipage}{0.19\linewidth} \centering \small \textbf{TT-ALS} \end{minipage}
    \hfill
    \begin{minipage}{0.19\linewidth} \centering \small \textbf{TT-ALS (Slice)} \end{minipage}
    \hfill
    \begin{minipage}{0.19\linewidth} \centering \small \textbf{TT-FOA} \end{minipage}
    \hfill
    \begin{minipage}{0.19\linewidth} \centering \small \textbf{Online TT-ALS} \end{minipage}
    
    \vspace{1mm}

    \rotatebox{90}{\small \textbf{Low Rank}}\hspace{2mm}%
    \begin{subfigure}[c]{0.18\linewidth}
        \includegraphics[width=\linewidth]{results/fig/gray/office_original_frame350.pdf}
    \end{subfigure}
    \hfill
    \begin{subfigure}[c]{0.18\linewidth}
        \includegraphics[width=\linewidth]{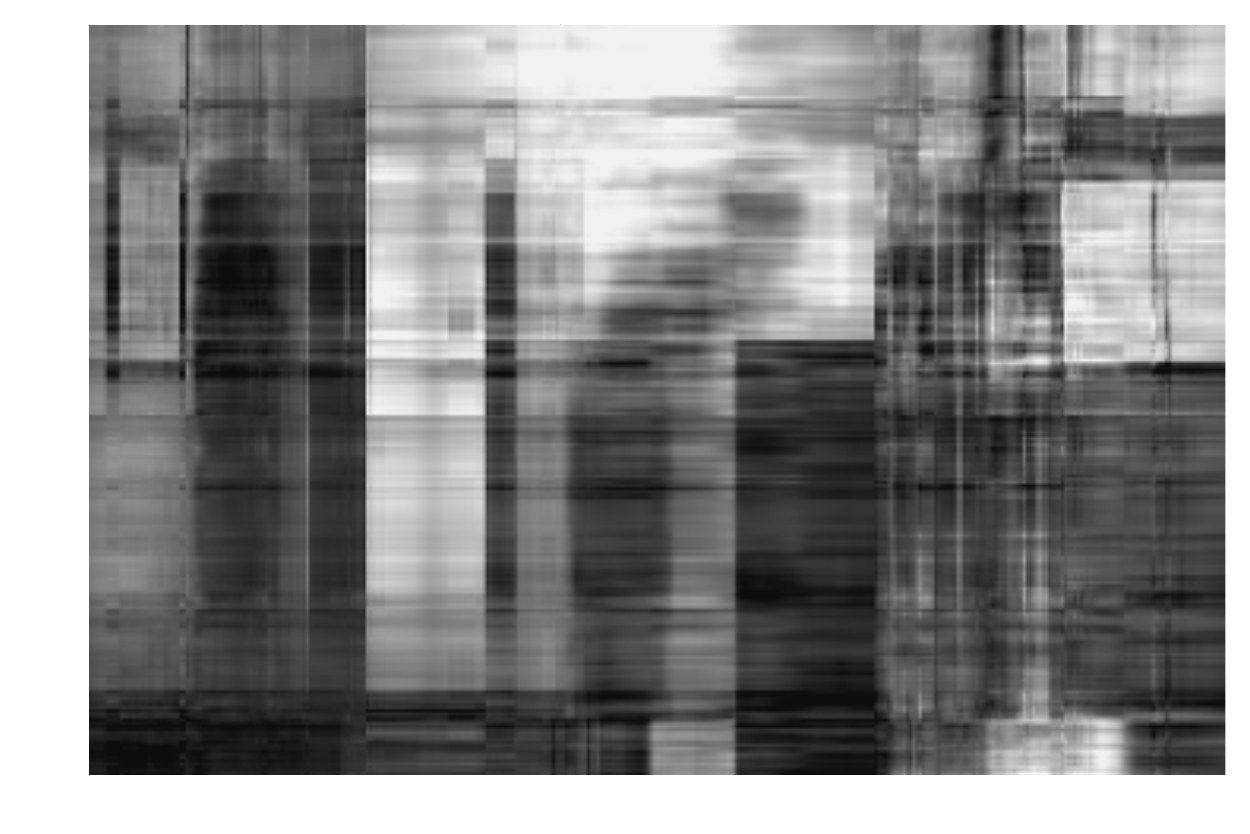}
    \end{subfigure}
    \hfill
    \begin{subfigure}[c]{0.18\linewidth}
        \includegraphics[width=\linewidth]{results/fig/gray/office_batchSliceALS_est_10_frame350.pdf}
    \end{subfigure}
    \hfill
    \begin{subfigure}[c]{0.18\linewidth}
        \includegraphics[width=\linewidth]{results/fig/gray/office_FOA_est_10_frame350.pdf}
    \end{subfigure}
    \hfill
    \begin{subfigure}[c]{0.18\linewidth}
        \includegraphics[width=\linewidth]{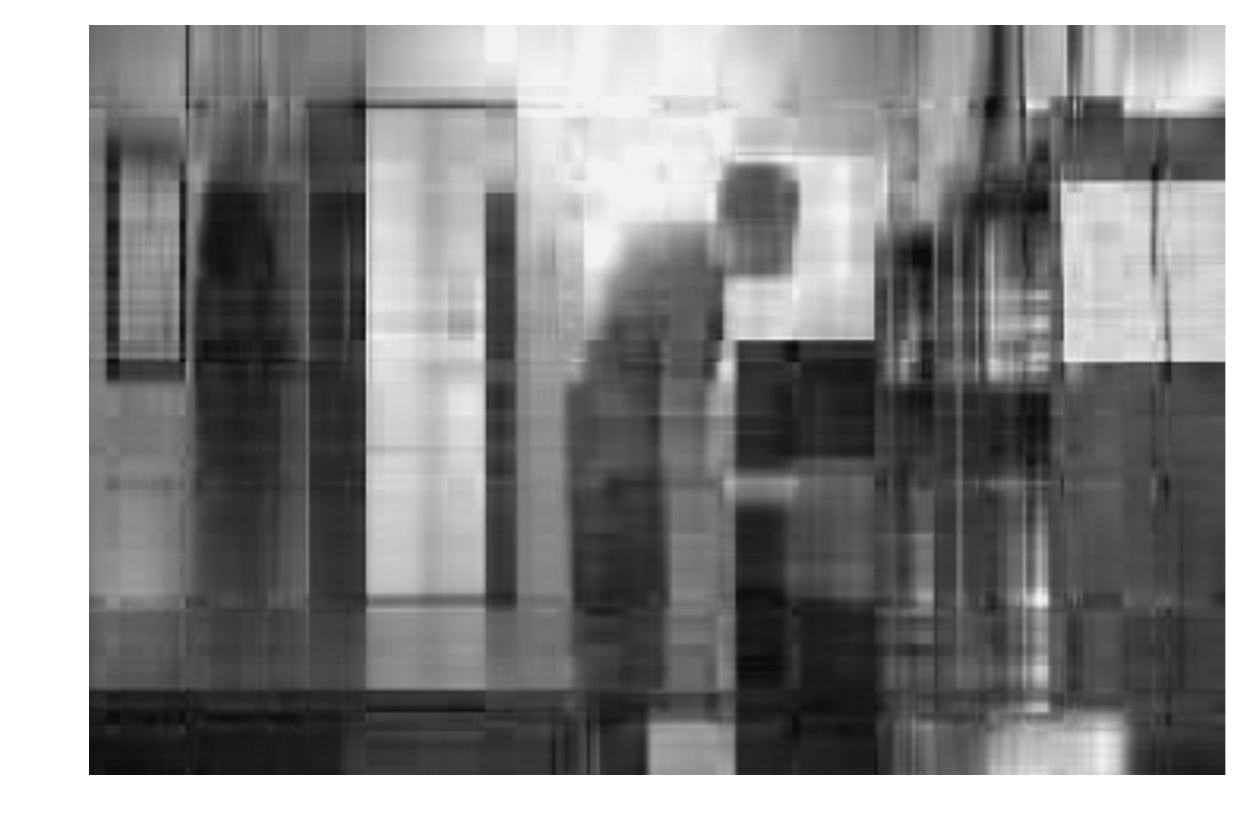}
    \end{subfigure}

    \vspace{2mm}

    \rotatebox{90}{\small \textbf{High Rank}}\hspace{2mm}%
    \begin{subfigure}[c]{0.18\linewidth}
        \includegraphics[width=\linewidth]{results/fig/gray/office_original_frame350.pdf}
    \end{subfigure}
    \hfill
    \begin{subfigure}[c]{0.18\linewidth}
        \includegraphics[width=\linewidth]{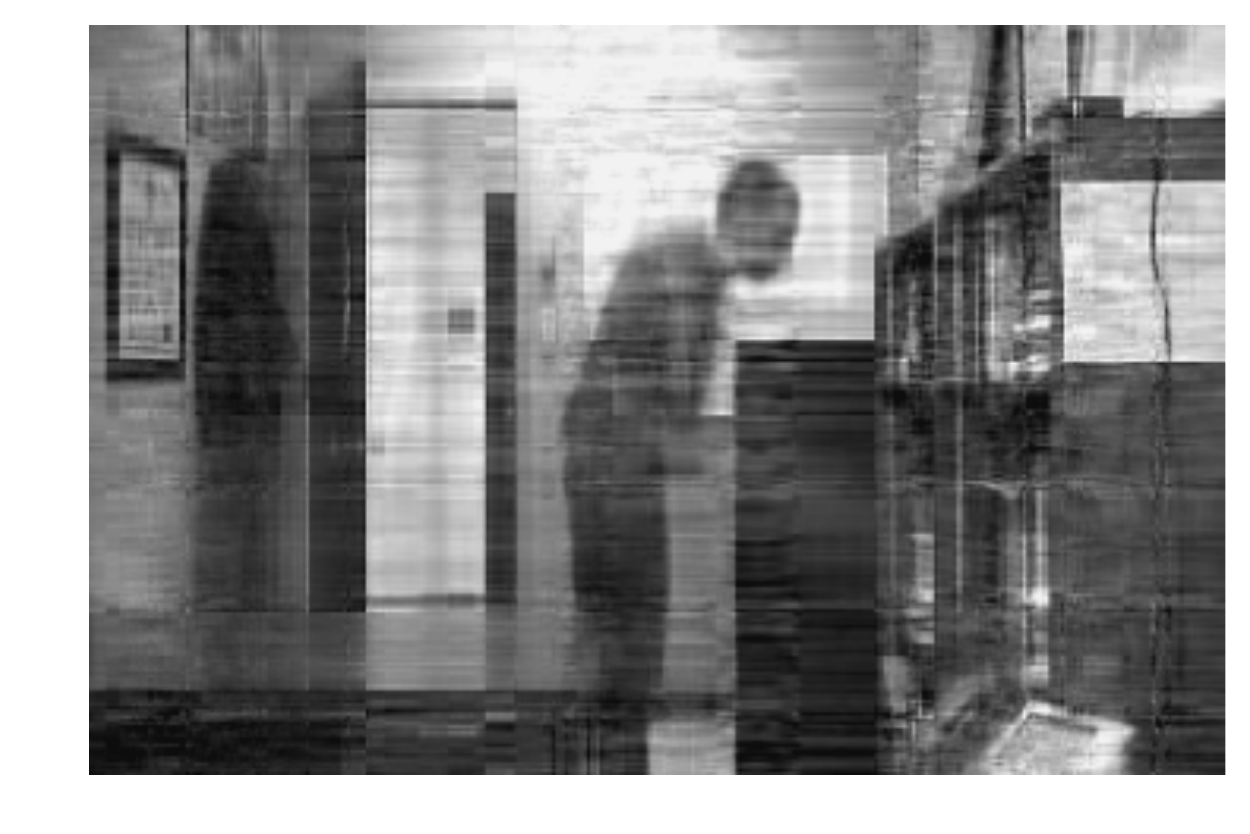}
    \end{subfigure}
    \hfill
    \begin{subfigure}[c]{0.18\linewidth}
        \includegraphics[width=\linewidth]{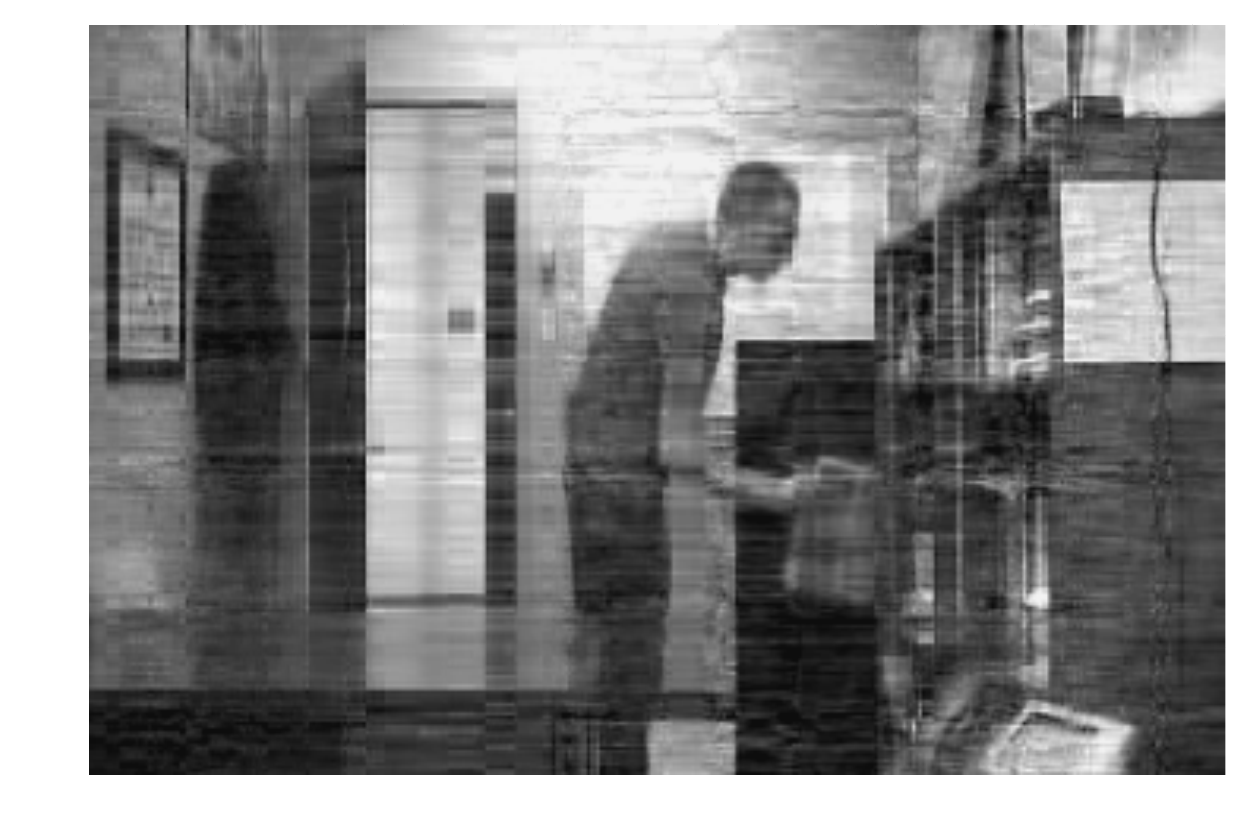}
    \end{subfigure}
    \hfill
    \begin{subfigure}[c]{0.18\linewidth}
        \includegraphics[width=\linewidth]{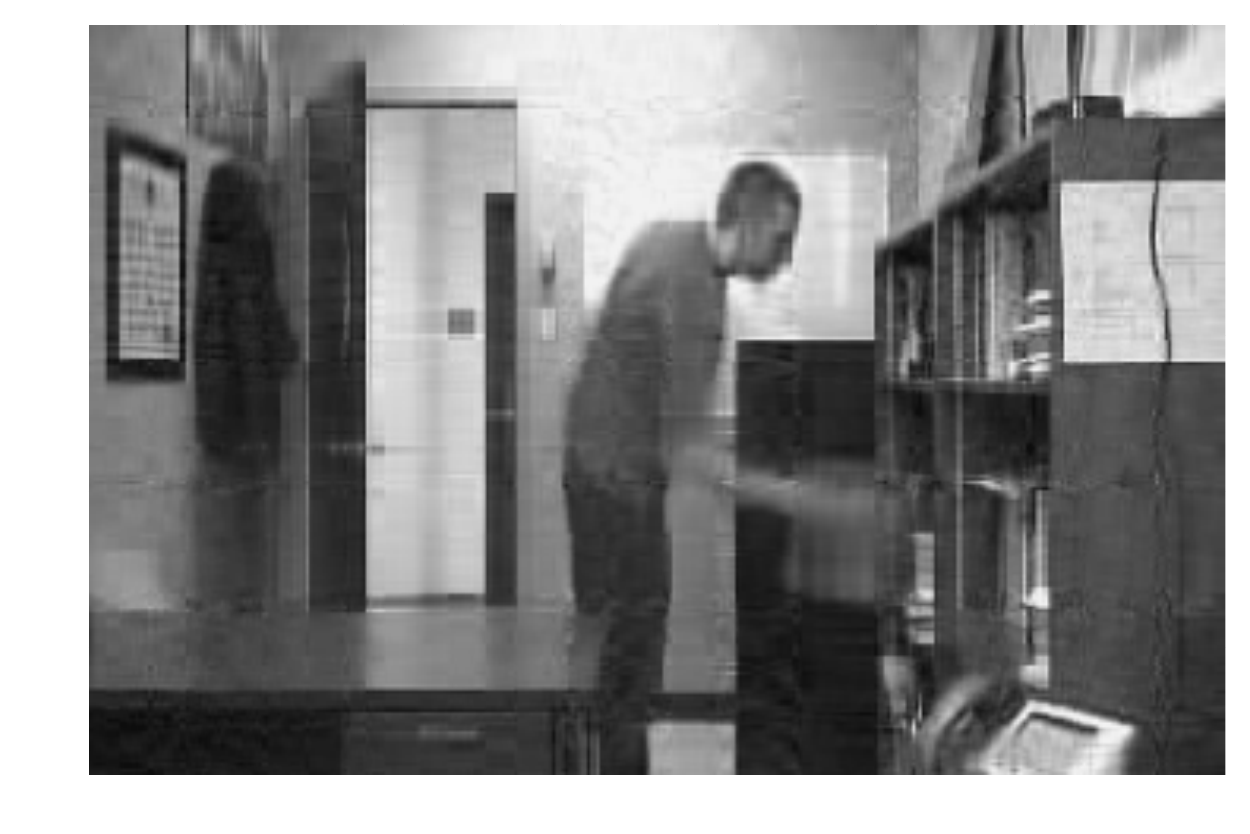}
    \end{subfigure}
    \hfill
    \begin{subfigure}[c]{0.18\linewidth}
        \includegraphics[width=\linewidth]{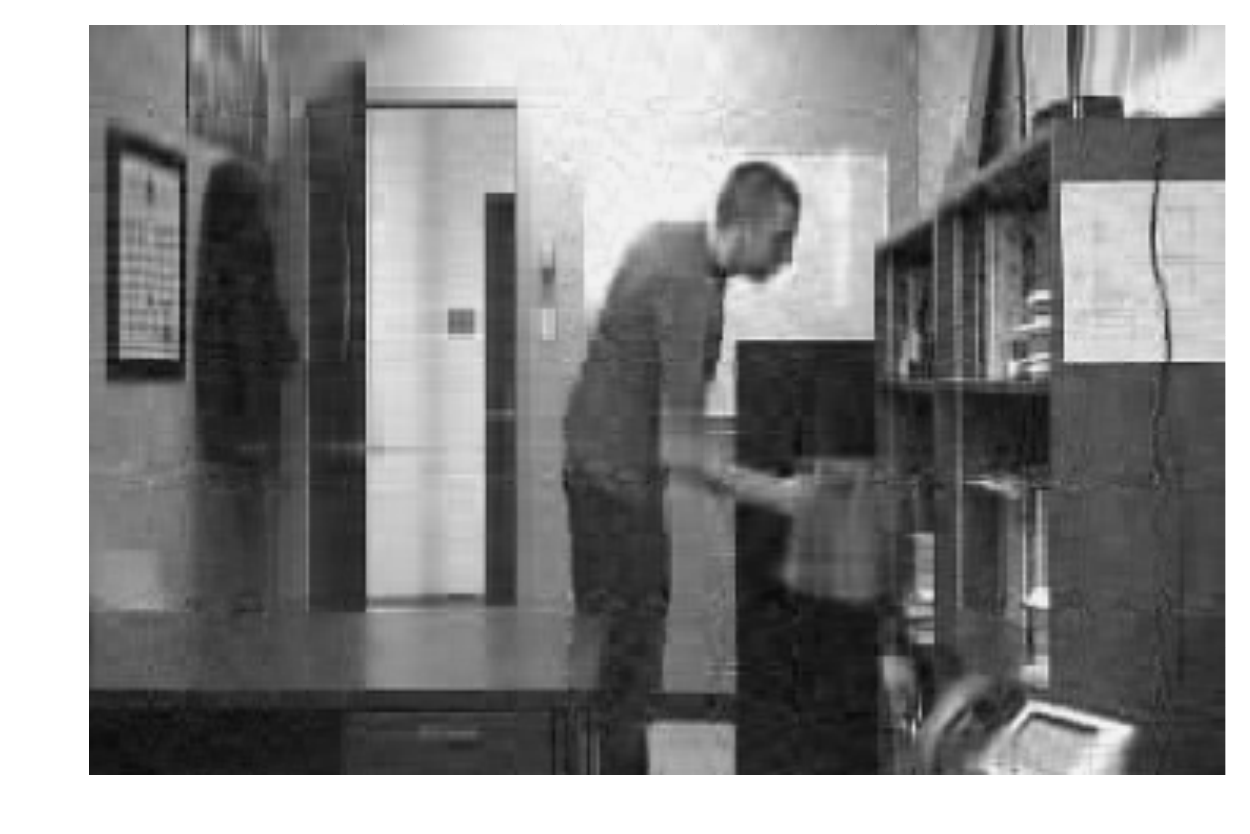}
    \end{subfigure}

    \caption{Visual comparison of reconstructed images at frame 350 of the grayscale video. The top row corresponds to the low-rank setting ($r=(10,10)$), and the bottom row corresponds to the high-rank setting ($r=(30,30)$).}
    \label{fig:visual_comparison_gray}
\end{figure*}

\begin{figure*}[t]
    \centering
    \begin{minipage}{0.19\linewidth} \centering \small \textbf{Original} \end{minipage}
    \hfill
    \begin{minipage}{0.19\linewidth} \centering \small \textbf{TT-ALS} \end{minipage}
    \hfill
    \begin{minipage}{0.19\linewidth} \centering \small \textbf{TT-ALS (Slice)} \end{minipage}
    \hfill
    \begin{minipage}{0.19\linewidth} \centering \small \textbf{TT-FOA} \end{minipage}
    \hfill
    \begin{minipage}{0.19\linewidth} \centering \small \textbf{Online TT-ALS} \end{minipage}
    
    \vspace{1mm}

    \rotatebox{90}{\small \textbf{Low Rank}}\hspace{2mm}%
    \begin{subfigure}[c]{0.18\linewidth}
        \includegraphics[width=\linewidth]{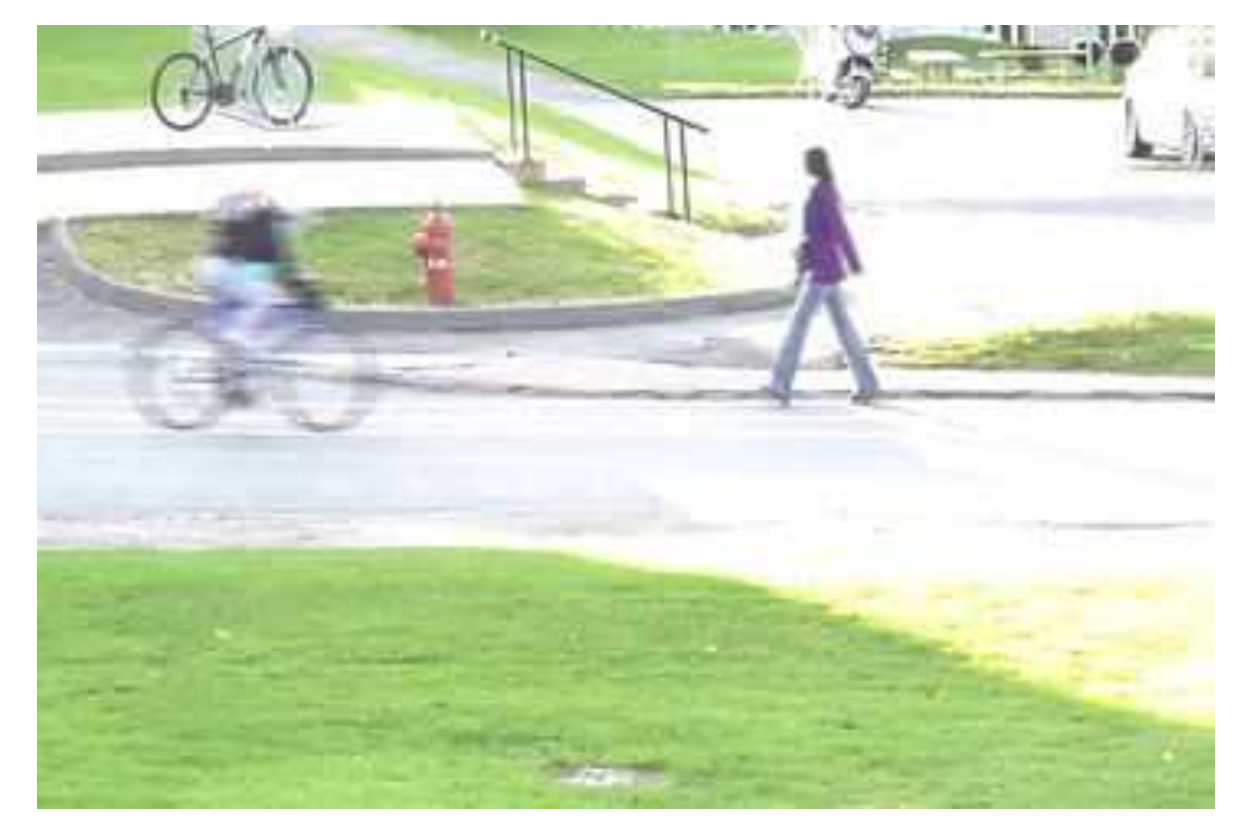}
    \end{subfigure}
    \hfill
    \begin{subfigure}[c]{0.18\linewidth}
        \includegraphics[width=\linewidth]{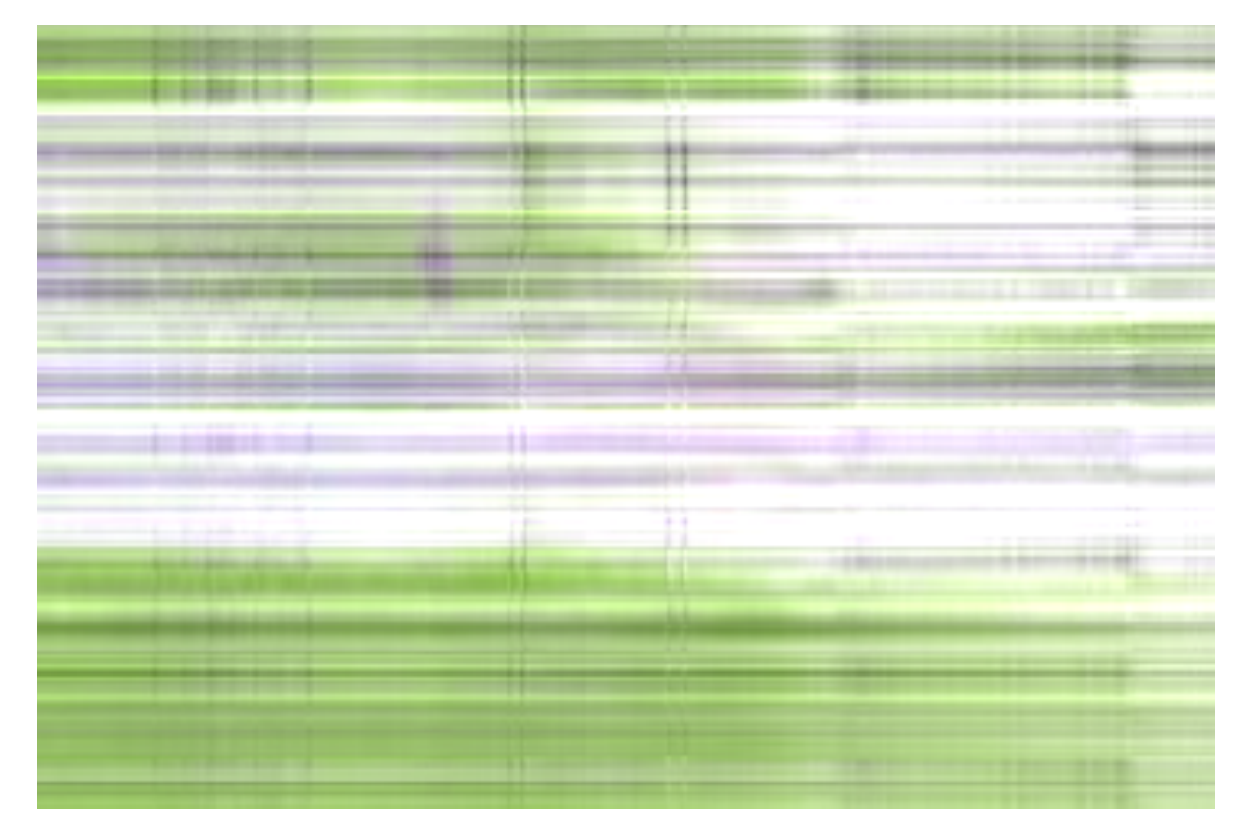}
    \end{subfigure}
    \hfill
    \begin{subfigure}[c]{0.18\linewidth}
        \includegraphics[width=\linewidth]{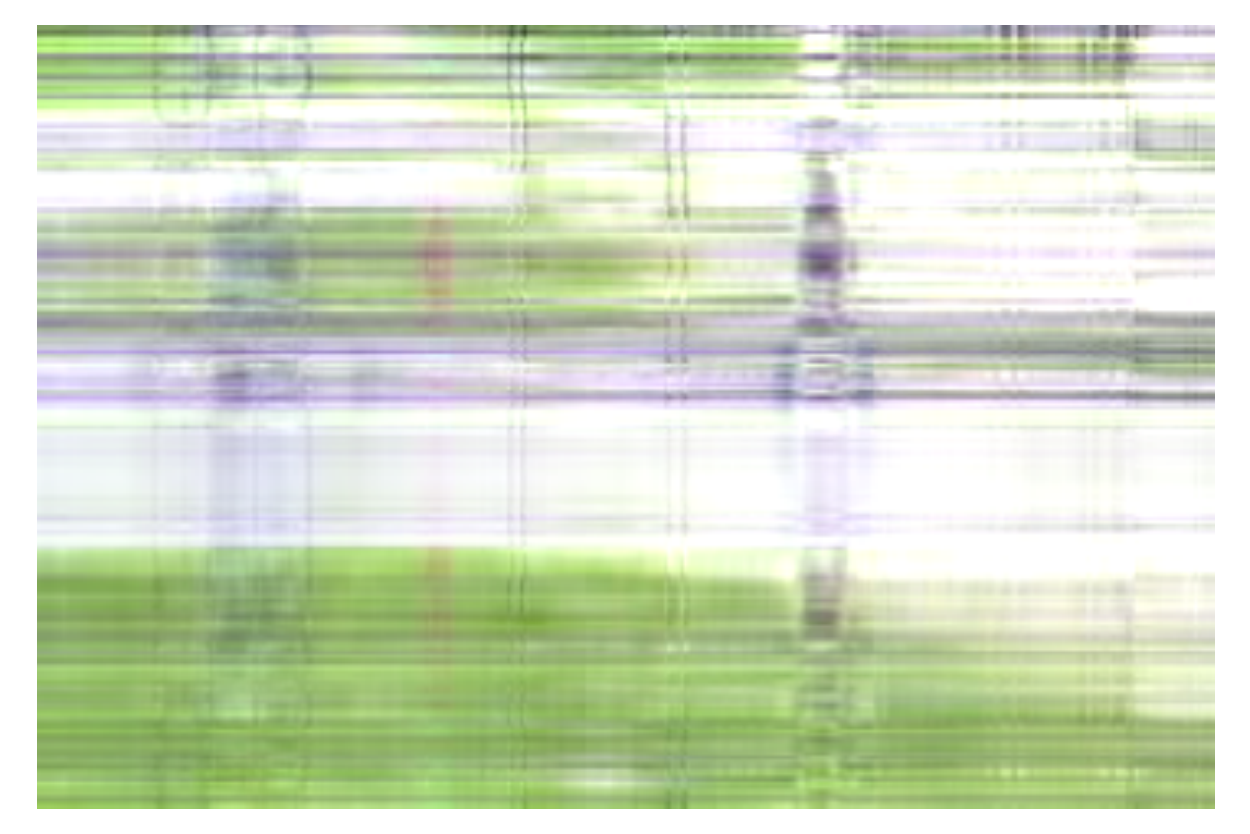}
    \end{subfigure}
    \hfill
    \begin{subfigure}[c]{0.18\linewidth}
        \includegraphics[width=\linewidth]{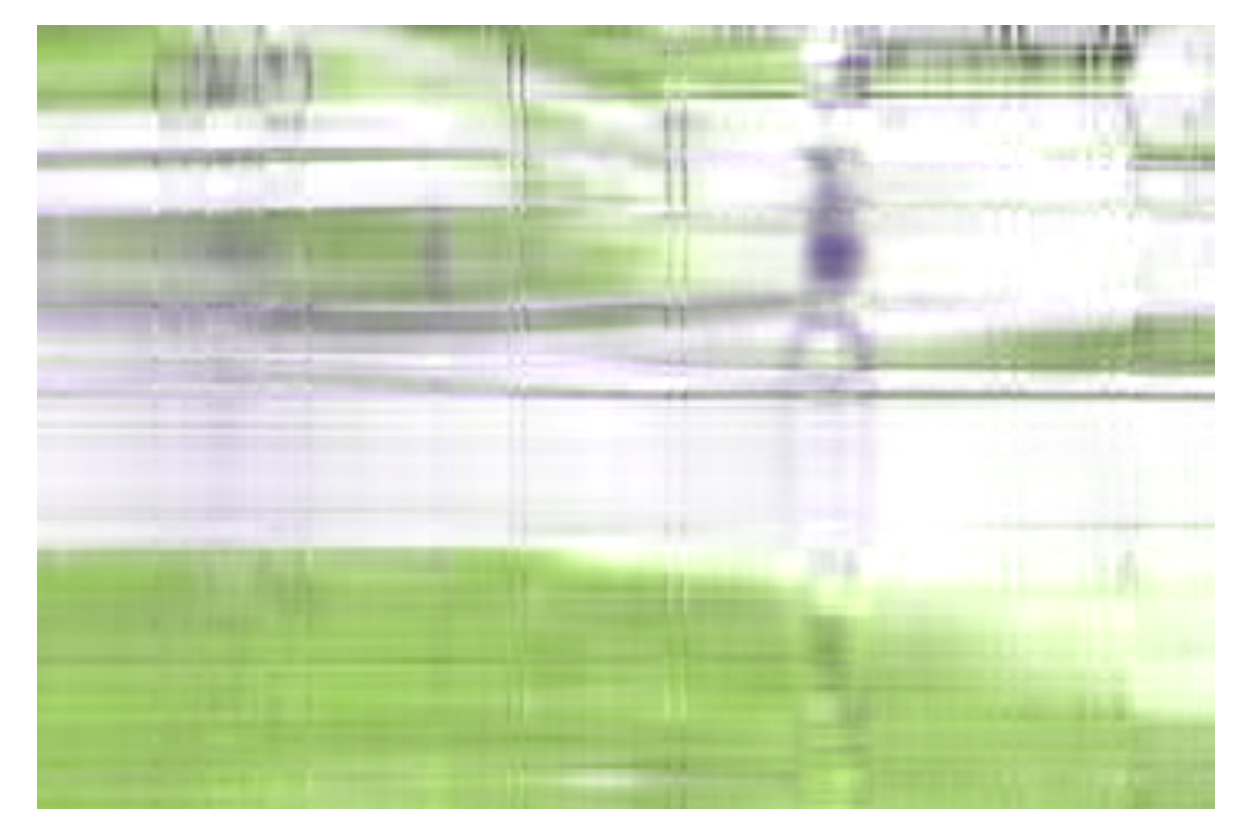}
    \end{subfigure}
    \hfill
    \begin{subfigure}[c]{0.18\linewidth}
        \includegraphics[width=\linewidth]{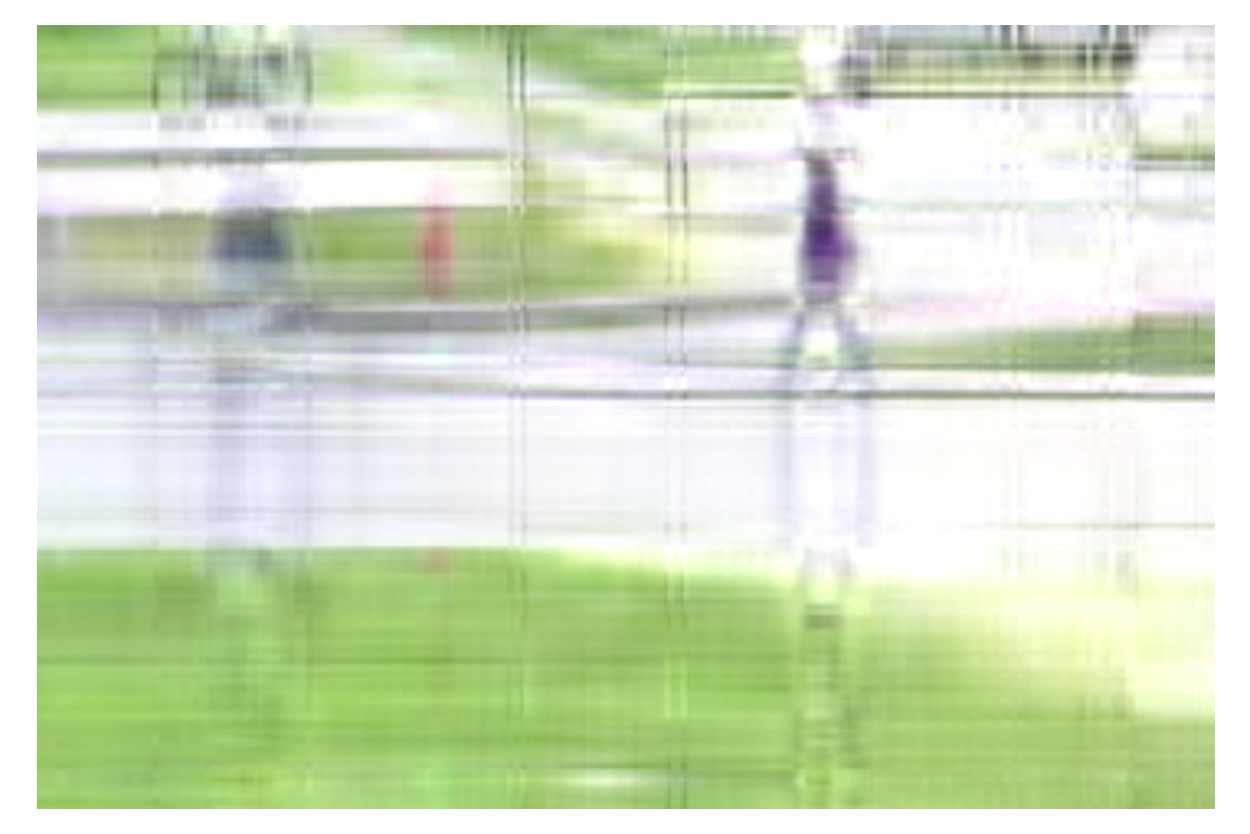}
    \end{subfigure}
    
    \vspace{2mm}

    \rotatebox{90}{\small \textbf{High Rank}}\hspace{2mm}%
    \begin{subfigure}[c]{0.18\linewidth}
        \includegraphics[width=\linewidth]{results/fig/color/pedestrians_original_frame220.pdf}
    \end{subfigure}
    \hfill
    \begin{subfigure}[c]{0.18\linewidth}
        \includegraphics[width=\linewidth]{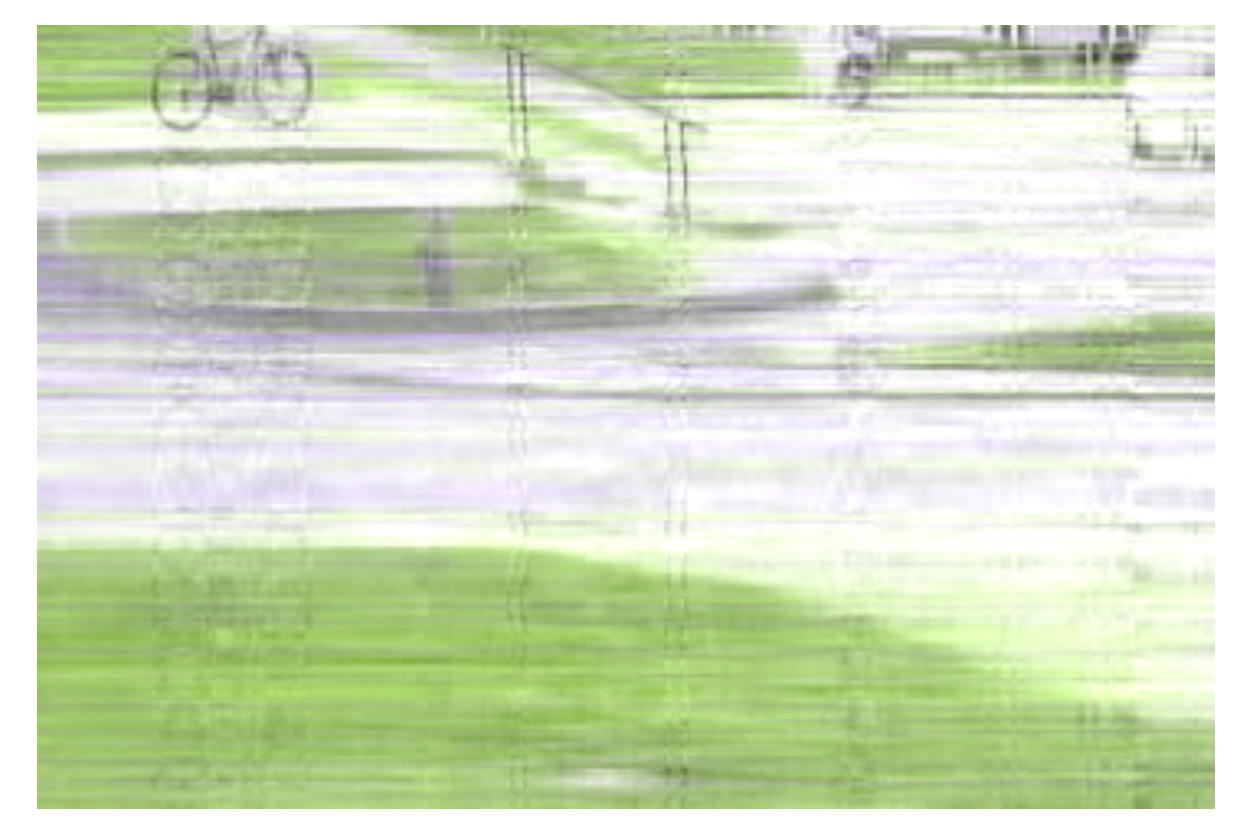}
    \end{subfigure}
    \hfill
    \begin{subfigure}[c]{0.18\linewidth}
        \includegraphics[width=\linewidth]{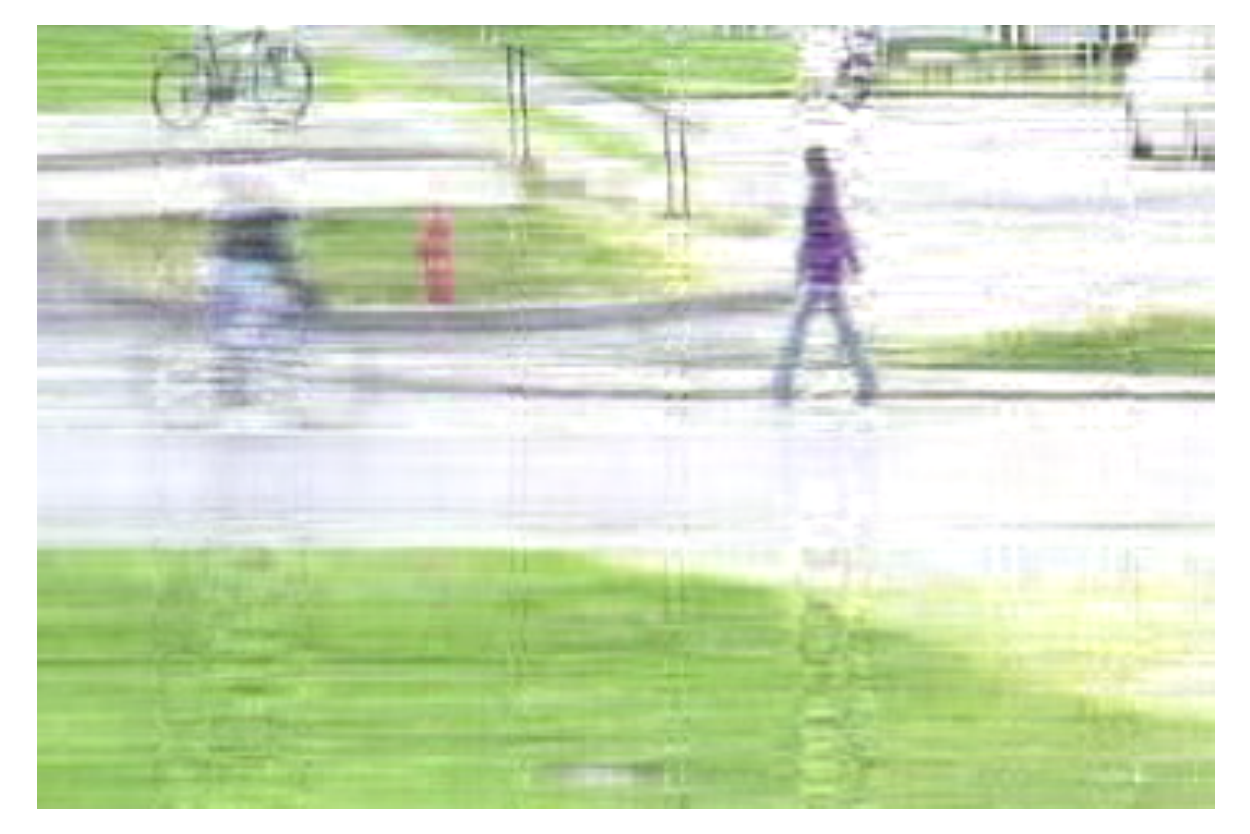}
    \end{subfigure}
    \hfill
    \begin{subfigure}[c]{0.18\linewidth}
        \includegraphics[width=\linewidth]{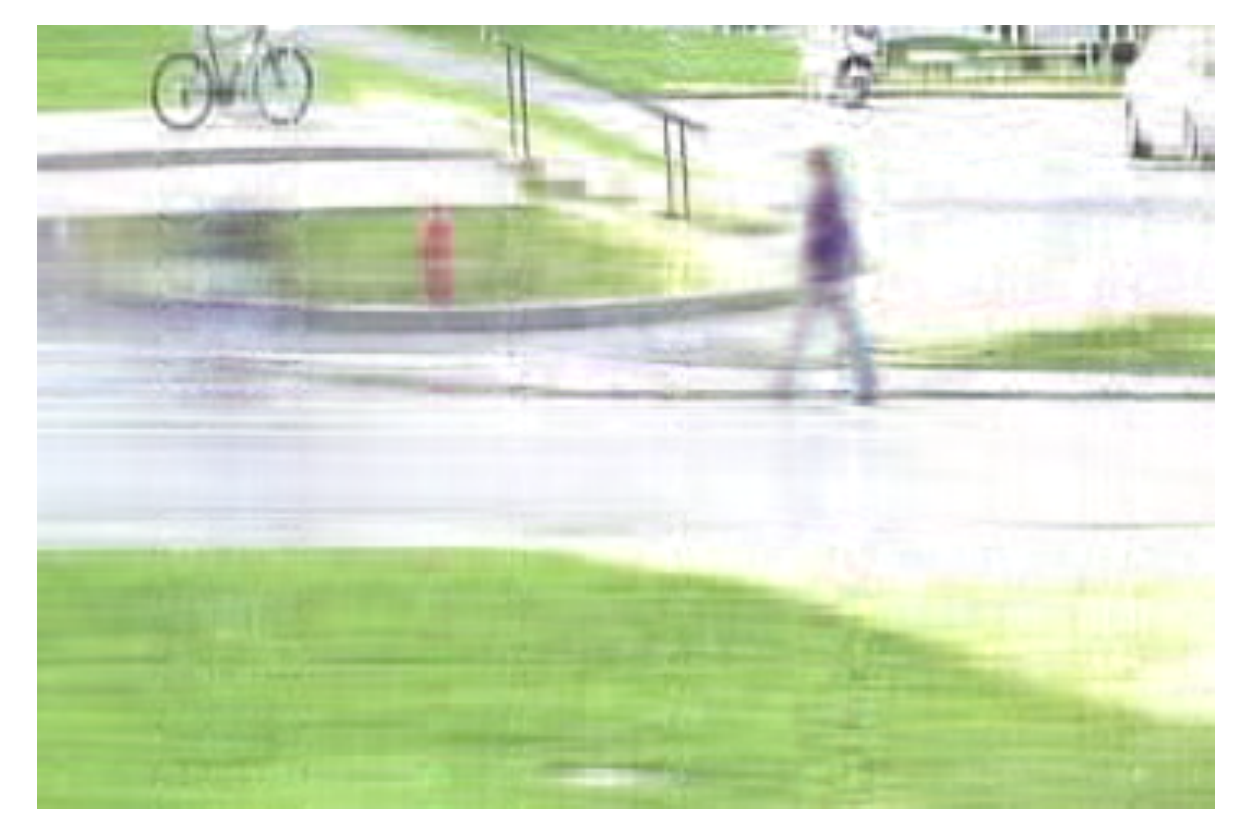}
    \end{subfigure}
    \hfill
    \begin{subfigure}[c]{0.18\linewidth}
        \includegraphics[width=\linewidth]{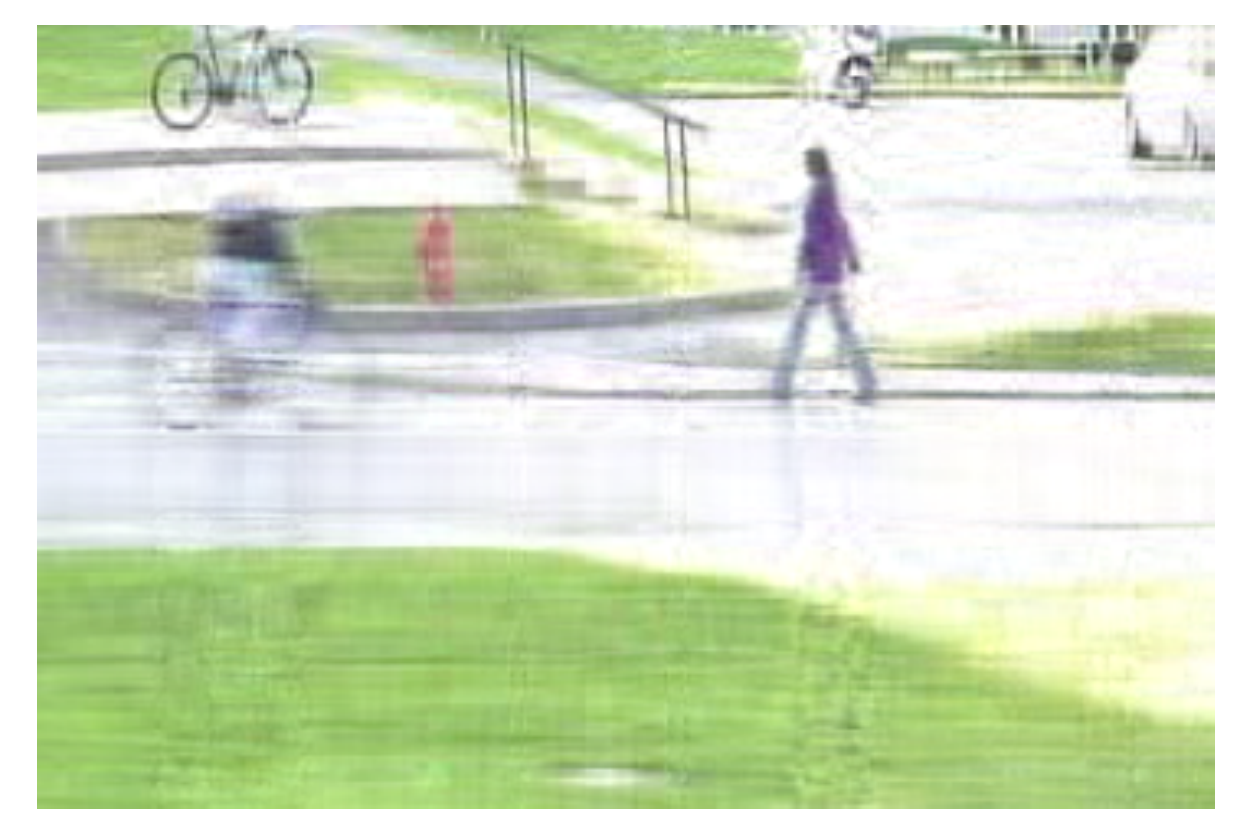}
    \end{subfigure}
    \caption{Visual comparison of reconstructed images at frame 220 of the color video. The top row corresponds to the low-rank setting ($r=(10,3,3)$), and the bottom row corresponds to the high-rank setting ($r=(30,3,3)$).}
    \label{fig:visual_comparison_color}
\end{figure*}

\paragraph{Tracking Stability}
Finally, beyond average metrics and visual snapshots, there is a fundamental difference in temporal tracking behavior between TT-FOA and our approach.
Because TT-FOA relies on randomized sampling and first-order approximations via Recursive Least Squares, it requires a warm-up period, often taking tens to hundreds of frames for the approximation error to fully converge.
By deterministically tracking the optimal subspace via exact single-sweep orthogonal updates, our method accurately tracks the streaming data from the initial frames without a warm-up delay, ensuring its reliability in real-time streaming applications.

\subsubsection{Latency Comparison with Deep Learning Paradigms}
Recent years have seen a surge in deep learning-based approaches to tensor representation.
To position our algebraic method within this broader landscape, we compared its computational latency against OFTD \citep{OFTD}, a state-of-the-art neural network-based continuous functional approximation framework.

Comparing an algebraic tensor decomposition directly with a neural network requires careful consideration, as they possess fundamentally different architectures.
To ensure a fair evaluation, we aligned our experimental setup with the constraints of the baseline.
Since the official implementation of OFTD is explicitly designed for third-order tensors, we conducted this comparison using the streaming grayscale video data ($240 \times 360 \times 500$) rather than the fourth-order color video.
Furthermore, because CP decomposition (utilized in OFTD) and Tensor Train decomposition have fundamentally different algebraic structures, directly matching rank values or parameter counts does not strictly establish equivalent conditions.
Therefore, we evaluated OFTD across a comprehensive range of CP-ranks ($R = 10, 20, 40$) and compared the outcomes with our proposed method under standard TT-ranks.

\begin{table}[t]
    \centering
    \caption{Quantitative comparison of approximation error and computational latency between OFTD and the proposed method on the grayscale video sequence.}
    \label{tab:comparison_oftd}
        \begin{tabular}{l l c r}
            \toprule
            \textbf{Method} & \textbf{Rank Setting} & \textbf{Average RE} $\downarrow$ & \textbf{Average Time / Frame} $\downarrow$ \\
            \midrule
            OFTD & CP-Rank $R=10$ & 0.131 & 26.16 s \ \ (26165.3 ms) \\
            OFTD & CP-Rank $R=20$ & 0.093 & 21.54 s \ \ (21548.8 ms) \\
            OFTD & CP-Rank $R=40$ & 0.065 & 22.48 s \ \ (22489.5 ms) \\
            \midrule
            Online TT-ALS (Ours) & TT-Rank $r=(10, 10)$ & 0.128 & 0.0027 s \ \ (2.756 ms) \\
            Online TT-ALS (Ours) & TT-Rank $r=(30, 30)$ & 0.068 & 0.0245 s \ \ (24.485 ms) \\
            \bottomrule
        \end{tabular}
\end{table}

The quantitative results and execution times are summarized in Table \ref{tab:comparison_oftd}.
The experimental outcomes highlight a fundamental difference in the computational paradigms.
While OFTD achieves reasonable approximation accuracy across various rank settings, its processing time consistently requires tens of seconds per frame.
This massive latency stems from its continuous learning mechanism.
Specifically, updating the network weights for each newly arriving frame requires many iterative forward and backward passes.

In contrast, our proposed Online TT-ALS completes the state update in $2.7$ to $24.5$ milliseconds per frame, achieving a speedup by a factor of $10^3$ to $10^4$.
Unlike deep learning models that require iterative backpropagation, our method achieves this efficiency through deterministic, single-sweep orthogonal updates.

Ultimately, these results clearly demonstrate that while neural implicit representations and streaming tensor completion frameworks are exceptionally powerful for imputing missing data, they operate under a completely different latency regime.
For the real-time processing of high-dimensional streaming data, the proposed framework provides a practical, robust, and scalable solution.

\subsection{Ablation Study on Orthogonality}
A core theoretical contribution of our proposed framework is the exact subspace tracking enabled by the sequential orthogonalization of the TT cores.
To empirically validate the necessity of this mechanism, we conducted an ablation study using the \textit{continuousPan} sequence under the high-rank setting $r=(30, 3, 3)$.
We compared the full proposed method against a variant, denoted by ``w/o Ortho,'' where the LQ/QR orthogonalization steps were removed during the sequential updates.

As shown in Table \ref{tab:ablation_ortho}, while both methods yield identical approximation errors at the beginning of the sequence ($t=10$), the performance of the variant without orthogonality degrades substantially as time progresses.
By $t=500$, the relative error of the non-orthogonal version increases significantly to $0.203$.
This degradation is driven by numerical instability and the resulting ill-conditioned Gram matrices that arise when core tensors are repeatedly updated without normalization. 

In contrast, the proposed method maintains a consistently low error ($\approx 0.100$) throughout the sequence.
This result indicates that enforcing orthogonality is an essential mechanism for preventing error accumulation and ensuring stable, long-term tracking in online streaming environments.

\begin{table}[t]
    \centering
    \caption{Ablation study on the effect of the orthogonality constraint for the \textit{continuousPan} sequence ($r=(30,3,3)$). The table presents the Relative Error (RE) over time $t$.}
    \label{tab:ablation_ortho}
        \begin{tabular}{l c c c}
            \toprule
            \textbf{Method} & $\mathbf{t=10}$ & $\mathbf{t=100}$ & $\mathbf{t=500}$ \\
            \midrule
            Proposed (w/o Ortho) & 0.097 & 0.147 & 0.203 \\
            Proposed (w/ Ortho) & 0.097 & 0.100 & 0.100 \\
            \bottomrule
        \end{tabular}
\end{table}

\section{Conclusion}
\label{sec:conclusion}
In this paper, we proposed Online TT-ALS, an algebraic algorithm for the low-latency compression and tracking of high-dimensional streaming tensor data.
We established the theoretical foundation for our approach, proving that enforcing orthogonal gauge constraints guarantees monotonic decrease of the local objective function and temporal smoothness.
Furthermore, our computational analysis demonstrated that maintaining this orthogonal subspace decouples the optimization process, resulting in an $\mathcal{O}(I^{n-1} r)$ computational complexity.
By scaling linearly rather than quadratically with respect to the TT-rank, our deterministic method reduces computational overhead compared to existing first-order approximations, avoiding the need for randomized sampling while ensuring stable subspace tracking.

Experimental results support these theoretical and computational advantages.
Our proposed framework exhibited scalability on high-order streaming tensors, successfully operating where conventional batch methods fail due to memory exhaustion.
On real-world video sequences, including those with dynamic camera motion, Online TT-ALS achieved high approximation accuracy and perceived visual quality (e.g., VMAF and SSIM) from the initial frames without requiring a warm-up period.
Furthermore, compared to deep learning-based continuous functional approximations, our algebraic method achieves speedups of several orders of magnitude by avoiding the overhead of iterative backpropagation.
Additionally, our ablation study empirically confirmed that the orthogonality constraint is essential for preventing error accumulation and ill-conditioned Gram matrices, thereby ensuring long-term numerical stability in online environments.

For future work, we plan to explore adaptive TT-rank determination mechanisms.
By dynamically adjusting the ranks to varying data complexities during online processing, we aim to enhance our method's efficiency for deployments on memory- and resource-constrained devices.
Additionally, we intend to extend the application of our framework to diverse types of high-dimensional streaming data beyond video processing.

\section*{Code Availability}
The source code and Julia notebooks used to reproduce the experimental results presented in this paper are publicly available at \url{https://github.com/hirokin0919/Online-TT-ALS}.

\section*{Acknowledgments}
This work was supported by JSPS KAKENHI Grant Numbers 21K18301, 24K02951, 24K00540, 25H00449, and 25K21806.

\bibliography{main}
\bibliographystyle{unsrtnat}
\end{document}